\documentclass[10pt, oneside,reqno]{amsart}\usepackage{amscd}
\usepackage{amssymb}
\usepackage{latexsym}
\usepackage{graphicx}
\usepackage[parfill]{parskip} 
\usepackage{amsmath}
\usepackage{amsfonts}
\usepackage{amssymb}
\usepackage{latexsym}
\usepackage{graphicx}
\usepackage{enumerate}
\usepackage{bbm}
\usepackage{verbatim}
\usepackage{bbold}
\usepackage[utf8]{inputenc}
\usepackage{fourier}

\usepackage{bbm}

\usepackage[final]{hyperref}
\hypersetup{colorlinks=true, linkcolor=blue, anchorcolor=blue, citecolor=red, filecolor=blue, menucolor=blue, pagecolor=blue, urlcolor=blue}
\newcommand{\Z}{\mathbb{Z}}

\newcommand{\abs}[1]{\vert #1 \vert}

\newcommand{\Bigabs}[1]{\Bigl\vert #1 \Bigr\vert}

\newcommand{\norm}[1]{\left\Vert #1 \right\Vert}

\newcommand{\C}{\mathbb{C}}

\newcommand{\R}{\mathbb{R}}

\newcommand{\innerprod}[2]{\langle \, #1 , #2 \, \rangle}

\newcommand{\angles}[1]{\langle #1 \rangle}

\DeclareMathOperator{\supp}{supp}

\newtheorem{theorem}{Theorem}
\newtheorem{proposition}{Proposition}
\newtheorem{lemma}{Lemma}
\newtheorem{corollary}{Corollary}

\theoremstyle{definition}
\newtheorem{definition}{Definition}

\theoremstyle{remark}
\newtheorem{remark}{Remark}

\numberwithin{equation}{section}
\begin{document}

 %//////////////////////////////////////////////////////////////////////////////////////////////////

\title[Small data scattering for Dirac type equation in 3D]{Small data scattering for a cubic Dirac equation with Hartree type nonlinearity in $\R^{1+3}$}

\author{Achenef Tesfahun}

\email{achenef@gmail.com}
 
\address{Department of Mathematics\\
University of Bergen\\
PO Box 7803\\
5020 Bergen\\ Norway}


\begin{abstract}
We prove that the initial value problem
for the Dirac equation 
$$
  \left (  -i\gamma^\mu \partial_\mu + m \right) \psi= \left( \frac{e^{- |x|}}{|x|} \ast ( \overline \psi \psi)\right)  \psi  \quad   \text{in } \ \R^{1+3}
  $$
is globally well-posed and the solution scatters to free waves asymptotically as $t \rightarrow \pm \infty$,
if we start with initial data that is small in $H^s$ for $s>0$. This
is an almost critical well-posedness result in the sense
that $L^2$ is the critical space for the equation.
The main ingredients in the proof are Strichartz estimates, space-time bilinear  
null-form estimates for free waves in $L^2$, and an application of the $U^p$ and $V^p$-- function spaces.
\end{abstract}

\maketitle
\setcounter{tocdepth}{1}
\tableofcontents

%%%%%%%%%%%%%%%%%%%%%%%%%%%%%%%%%%%%%%%%%

\section{Introduction}

\subsection{Preliminary}
We consider the initial value problem for the nonlinear Dirac equation with Hartree type nonlinearity
 \begin{equation}\label{Dirac1}
\left\{
\begin{aligned}  \left ( -i\gamma^\mu \partial_\mu + m \right) \psi &= \left(V\ast ( \overline \psi \psi)\right)  \psi  \quad   \text{in } \ \R^{1+3}, 
   \\
       \psi(0,\cdot)&=\psi_0 \in H^s(\R^3),
    \end{aligned}
\right.
\end{equation}
where the unknowns are a spinor field $\psi(t,x)$
regarded as a column vector in $\C^4$;  $m\ge 0$ is a mass parameter; 
$V(x)= |x|^{-1} e^{ -|x|}  $ is the Yukawa potential;  $\gamma^\mu \partial_\mu=\gamma^{0}\partial_t + \sum_{j=1}^3 \gamma^{j}\partial_{x_j}$, where $\{ \gamma^\mu \}_{\mu = 0}^3$ are the $4 \times 4$ Dirac
matrices, given in $2 \times 2$ block form by
$$
  \gamma^{0} = \begin{pmatrix}
    I & 0  \\
    0 & -I
  \end{pmatrix},
  \qquad
  \gamma^{j} = \begin{pmatrix}
    0 & \sigma^{j}  \\
    -\sigma^{j} & 0
  \end{pmatrix},
$$
where
$$
  \sigma^{1} =
  \begin{pmatrix}
    0 & 1  \\
    1 & 0
  \end{pmatrix},
  \qquad
  \sigma^{2} =
  \begin{pmatrix}
    0 & -i  \\
    i & 0
  \end{pmatrix},
  \qquad
  \sigma^{3} =
  \begin{pmatrix}
    1 & 0  \\
    0 & -1
  \end{pmatrix}
$$
are the Pauli matrices; $ \overline \psi =\psi^\dagger \gamma^0 $, where $\psi^\dagger$ denotes the
the conjugate transpose, hence
$$
  \overline \psi \psi =\psi^\dagger \gamma^0  \psi =\innerprod{\gamma^0\psi}{ \psi}  \equiv \abs{\psi_1}^2 +
  \abs{\psi_2}^2 - \abs{\psi_3}^2 - \abs{\psi_4}^2,
$$
where $\psi_1,\dots,\psi_4$ are the components of $\psi$. Finally, $H^s$ is the Sobolev space of order $s$.

Equation \eqref{Dirac1} with a Coulomb potential, i.e., $V(x)= |x|^{-1} $,  and a quadratic term
 $|\psi|^2$ replacing $\overline \psi \psi$ was derived by Chadam 
 and Glassey \cite{CG76} by uncoupling 
 the Maxwell-Dirac equations under the assumption of vanishing magnetic field.
 They also conjectured in the same paper \cite[see  pp. 507]{CG76} 
  that \eqref{Dirac1} with a Yukawa potential $V$ can be derived by uncoupling 
 the Dirac-Klein-Gordon equations:
\begin{equation} \label{DKG}\tag{DKG}
\begin{cases}
\left ( -i\gamma^\mu \partial_\mu + m \right) \psi =\phi \psi,\\
(\partial_t^2-\Delta+M^2)\phi= \overline \psi \psi.
\end{cases}
\end{equation}
 
 Now if we assume that
the scalar field $\phi$ is a standing wave of the form $\phi(t,x)=e^{i c t}\varphi(x)$ with $|c|< M$ (see e.g. \cite{EGS96, CY19}), the Klein-Gordon part of \eqref{DKG} becomes
\begin{equation*}
 \left(-\Delta+ \left(M^2-c^2\right) \right)\phi= \overline \psi \psi,
\end{equation*}
 whose solution is given by
\begin{equation*}
 \phi= V_{m_0} \ast ( \overline \psi \psi),
\end{equation*}
where $V_{m_0}= (4\pi |x|)^{-1} e^{ -m_0 |x|}  $ with $m_0=\sqrt{M^2-c^2} > 0$. Plugging $\phi= V_{m_0} \ast ( \overline \psi \psi)$ back into the first equation in \eqref{DKG} yields \eqref{Dirac1} with $V$ replaced by $V_{m_0}$. Nevertheless, the analysis of \eqref{Dirac1} with the Yukawa potential $V$ or $V_{m_0}$ is the same.
 
   The $L^2$-norm of the solution  for \eqref{Dirac1} is conserved:
\begin{align*}
 \int_{\R^3} |\psi(t,x)|^2 \, dx = \int_{\R^3} |\psi(0,x)|^2 \, dx. 
\end{align*}
In the massless case, $m=0$, \eqref{Dirac1} is invariant under the scaling 
 $$u(t,x) \mapsto u_\lambda(t,x) = \lambda^{\frac32} u(\lambda t, \lambda x)$$ for 
fixed $\lambda > 0$. This scaling symmetry leaves the $L^2$--norm invariant, and 
so equation \eqref{Dirac1} is \emph{$L^2$-critical}.
 
 A related equation that has been studied extensively is the boson star equation:
 \begin{equation}\label{BS}
   \left(- i\partial_t  + \sqrt{m^2-\Delta} \right)\, u= (V \ast |u|^2)u \quad \text{in } \R^{1+3}.
\end{equation}
The first well-posedness result for this equation with both the Coulomb and Yukawa potential was obtained by Lenzmann \cite{L07}  
for data in $H^s$ with $s \geq \frac12$, and later this was improved to $s>1/4$  by Lenzmann and Herr \cite{HL14}.  Concerning scattering theory Pusateri \cite{FP14} established a modified scattering result in the case of Coulomb potential (which is the most difficult case) for small initial data in some weighted Sobolev space. 
There are several well-posedness and scattering results 
for \eqref{BS} with potentials of the form 
$V(x)= |x|^{-a}$ for $a\in (1, 3)$;  see eg. \cite{CO06,CO07, CO08, COSS09, FP14}.  
  
  Convolution with a Yukawa potential is (up to a multiplicative constant) the Fourier-multiplier $ \left(1-\Delta \right)^{-1}$ with Fourier symbol $ \left(1+ |\xi|^2\right)^{-1}$ in $\R^3$ while Convolution with a Coulomb potential is (up to a multiplicative constant) the Fourier-multiplier $ \left(-\Delta \right)^{-1}$ with Fourier symbol $  |\xi|^{-2}$ in $\R^3$. Thus, both of these potentials  are smoothing operators, however, the Yukawa potential has an advantage over the Coulomb potential (and also potentials of the form $V(x)= |x|^{-a}$ for $a\in (0, 3)$) since the latter one is singular near the origin. 
  
Recently, Herr and the present author \cite{HT15} proved small data scattering for \eqref{BS} with $m>0$ and $s>\frac12$. Consequently, scattering is obtained for the 
nonlinear Dirac equation 
 \begin{equation}\label{Dirac-BS}
  \left (  -i\gamma^\mu \partial_\mu + m \right) \psi= \left(V \ast |\psi|^2\right)  \psi  \quad   \text{in } \ \R^{1+3}
\end{equation}
with Yukawa potential, $m>0$ and $s>\frac12$ (see \cite[Remark 1.2]{HT15})\footnote{Global well-posedness and scattering of \eqref{Dirac1} for $m>0$ and $s>1/2$ will also follow from \cite{HT15}. However, Theorem \ref{mainthm} improves this result to  $m\ge 0$ and $s>0$ .}. 
 Existence of weak solution for \eqref{Dirac-BS} 
 with a Yukawa potential in the massless case ($m=0$) was proved earlier by Dias and Figueira \cite{DF89-2, DF89-1}. There is also a small data scattering result due to Machihara and Tsutaya \cite{MT09} for \eqref{Dirac-BS} 
 with a potential 
$V(x)= |x|^{-a}$ for $a\in (2, 3)$, $m>0$ and $s>a/6 + 1/2$.

The key difference between \eqref{Dirac1} and \eqref{Dirac-BS} is the nonlinearity $\left(V\ast ( \overline \psi \psi)\right)  \psi $ contains a hidden null-structure while this structure is not present in $ \left(V \ast |\psi|^2\right)  \psi .$
In the present paper, we exploit this null-structure to obtain small data scattering for \eqref{Dirac1} for all $m\ge 0$ and $s>0$. To establish this result we first prove $L^2$- space-time bilinear null-form estimates for free waves and frequency localized quadrilinear estimates in $U^p$ and $V^p$--spaces. To the authors knowledge there is no prior well-posedness result for \eqref{Dirac1}.

Our main result is as follows.
\begin{theorem}\label{mainthm}
Let $m\ge 0$, $s>0$ and $\norm{ \psi_0}_{H^s} < \varepsilon$ for sufficiently small $\varepsilon>0$. Then the initial value problem \eqref{Dirac1} is globally well-posed and the solution $\psi$ scatter to free waves as $t\rightarrow \pm \infty$.
                   
\end{theorem}

\begin{remark}
\begin{enumerate}
\item After the submission of this paper the author has learned that similar scattering
result for \eqref{Dirac1} with a Yukawa potential and potentials of type 
$V(x)= |x|^{-a}$ 
 was independently proved by Yang \cite{CY19}.

\item  The critical case $s=0$ corresponds to initial data in $L^2(\R^3).$ The presence of the factor $\lambda_{\text{med}}^\delta$ ($\delta>0$ small) in the dyadic quadrilinear estimate in Lemma \ref{KeyLemma1} impedes us to perform the sum in Lemma \ref{lemma-summing} for $s=0$. However, if one is able to prove the estimate with the factor $\lambda_{\text{med}}^{-\delta}$ or $(\lambda_{\text{med}}/\lambda_{\text{max}})^\delta$ replacing $\lambda_{\text{med}}^\delta$, then it is possible to do the sum for $s=0$, and hence prove Theorem \ref{mainthm} for initial data in $L^2(\R^3)$. This may however require modifying the working spaces or proving more refined estimates.  In this paper, we do not claim that the quadrilinear estimate in Lemma \ref{KeyLemma1} is optimal.

\item Recently, after the submission of this paper, the author established a scattering result for \eqref{Dirac1} in $\R^{1+2}$ for $m>0$ and $s>0$ \cite{AT18}. 
\end{enumerate}
\end{remark}

\subsection{Reformulation of Theorem \ref{mainthm}}
We rewrite \eqref{Dirac1} in a slightly different
form by multiplying the equation by $\beta =
\gamma^0$:
\begin{equation}\label{Dirac2}
\left\{
\begin{aligned}  ( - i\partial_t  + \boldsymbol \alpha \cdot D +   m \beta) \psi &= \left(V \ast \innerprod{\beta\psi}{\psi}\right)\beta \psi  \quad   \text{in } \ \R^{1+3},
   \\
       \psi(0,\cdot)&=\psi_0 \in H^s(\R^3),
    \end{aligned}
\right.
\end{equation}
where $D=-i \nabla$  and $\boldsymbol \alpha=(\alpha^1, \alpha^2, \alpha^3)$ with $
  \alpha^{j} \equiv \gamma^{0}\gamma^{j}$.
  These  matrices 
 satisfy the following identities:
\begin{equation}\label{matrixidentities1}
\begin{split}
 &\beta^2 = (\alpha^j)^2 = I ,
\quad
  \alpha^j \beta = - \beta \alpha^j.
  \\
 & \alpha^{j} \alpha^{k} = - \alpha^{k} \alpha^{j} + 2 \delta^{jk} I,
  \end{split}
\end{equation}
where $ \delta^{jk} =1$ if $j=k$ and  $ \delta^{jk} =0$ if $j\neq k$.
Moreover, 
\begin{equation}\label{matrixidentities2}
\begin{split}
  \alpha^{j} \alpha^{k}
  &= \delta^{jk} I + i \epsilon^{jkl} S_l,
  \end{split}
\end{equation}
where
$\epsilon^{jkl} =1$ if $(j,k, l)$ is an even permutation of $(1,2,3)$, $\epsilon^{jkl}  = -1$ if
$(j,k,l)$ is an odd permutation of $(1,2,3)$ and $\epsilon^{jkl}  = 0$ otherwise, and
$$
S^{l} =
  \begin{pmatrix}
    \sigma^{l} & 0  \\
    0 & \sigma^{l}
  \end{pmatrix}.
$$

Following \cite{DFS06, BH13} we decompose the spinor $\psi$ relative to
a basis of the operator $\boldsymbol \alpha \cdot D +   m \beta$ whose symbol is $\boldsymbol \alpha \cdot \xi +   m  \beta$. Since
$(\boldsymbol \alpha \cdot \xi +   m  \beta)^2=(|\xi|^2+m^2)I$ , the eigenvalues are $\pm \angles{\xi}_m$, where 
$$\angles{\xi}_m=\sqrt{m^2+|\xi|^2}.$$
Now define
the projections
$$ \Pi^{\pm}_m(D) = \frac{1}{2} \left( I \pm \frac{1}{\angles{D}_m }[  \boldsymbol \alpha \cdot D + m \beta]
\right).$$ 
 Then we can decompose
\begin{equation}
\label{psi-split}
 \psi = \psi^+ + \psi^-, \quad \text{ where} \ \
\psi^\pm  = \Pi^\pm_m(D) \psi. 
\end{equation}

In view of the identities in \eqref{matrixidentities1}--\eqref{matrixidentities2} we have
\begin{equation}
\label{Pidentities}
 \Pi^{\pm }_m(D) \Pi^{\pm }_m(D) =\Pi^{\pm }_m(D) , \quad  \Pi^{\pm }_m(D) \Pi^{\mp }_m(D) =0
\end{equation}and 
\begin{equation}
\label{Pbetaidentities}
 \beta \Pi^{\pm }_m(D)  =\Pi^{\mp }_m(D) \beta \pm m\angles{D}_m^{-1}.
\end{equation}

Applying $\Pi^\pm_m(D)$ to \eqref{Dirac2} and using \eqref{psi-split}--\eqref{Pidentities} we obtain 
\begin{equation}\label{Dirac3}
\left\{
\begin{aligned}
  & \bigl( -i\partial_t  + \angles{D}_m \bigr) \psi^+ =  \Pi_m^+(D)\left[(V \ast\innerprod{\beta\psi}{\psi}) \beta \psi\right], 
 \\
   & \bigl( -i\partial_t  - \angles{D}_m \bigr) \psi^- =  \Pi_m^-(D)\left[(V \ast \innerprod{\beta\psi}{\psi} )\beta \psi\right]
\end{aligned}
\right.
\end{equation}
with initial data
\begin{equation}\label{Data3}
 \psi^\pm(0,\cdot)=\psi_0^\pm \in H^s(\R^3),
\end{equation}
where 
$$\psi_0^\pm=\Pi^\pm_m(D)\psi_0.$$
We denote by  
$
 S_m (\pm t) 
$ the solution propagators to the free Dirac equation:
$$ S_m (\pm t)f =e^{\mp it \angles{D}_m} f=\int_{\R^3} e^{\mp it \angles{\xi}_m}  e^{ i x  \cdot \xi} \widehat f(\xi ) \ d\xi.$$

 Now Theorem \ref{mainthm} reduces to the following:
\begin{theorem}\label{mainthm1}
Let $m\ge 0$, $s>0$ and $\norm{ \psi_0^\pm}_{H^s} < \varepsilon$ for sufficiently small $\varepsilon>0$. Then the IVP \eqref{Dirac3}--\eqref{Data3} is globally well-posed and the solutions $\psi^\pm$ scatter to free waves  as $t\rightarrow  \pm \infty$, i.e., there exist 
$(f_\pm, g_\pm) \in H^s \times H^s$ such that
                  $$
   \lim_{t\rightarrow  \infty} \norm{\psi^\pm(t)-S_m(\pm t) f_\pm}_{H^s}=0
                    $$
 and 
                  $$
   \lim_{t\rightarrow  -\infty} \norm{\psi^\pm(t)-S_m(\pm t) g_\pm}_{H^s}=0.
                    $$               
\end{theorem}

The rest of the paper is organized as follows. In Section \ref{secfuncspaces} 
we give some notation, define the $U^p$ and $V^p$-spaces and collect their properties. 
In Section \ref{LinBi-Est}, we collect some linear, convolution and bilinear estimates
for free solutions of Klein-Gordon equation. In Section \ref{secnull} we reveal the null structure in \eqref{Dirac3} and prove bilinear null-form estimates. In Section \ref{sec-ProofMainThm} we give the proof for Theorem \ref{mainthm1} after reducing it first to non-linear estimates. The proof for these non-linear estimates will be given in Section \ref{secKeyLemma1}. In Sections \ref{seclmI+-} and \ref{seclmBiest} we prove the convolution and bilinear estimates for free waves stated in Section \ref{LinBi-Est}.

%%%%%%%%%%%%%%%%%%%%%%%%%%%%%%%%%%%%%%%%%

%%%%%%%%%%%%%%%%%%%%%%%%%%%%%%%%%%%%%%%%%

\section{Notation and Function spaces}\label{secfuncspaces}

 \subsection{Notation}
In equations, estimates and summations the Greek letters $\mu$ and $\lambda$  are presumed to be dyadic with $\mu, \lambda >  0$, i.e., these
variables range over numbers of the form $ 2^k$ for $k \in \Z$.
In estimates we use $A\simeq B$ as shorthand for $A = CB$ and $A \lesssim B$ for $A \le C B$ for some constant $C>0$ whereas we use $A\ll B$ for $A\le C^{-1} B$ for some constant $C\gg1 $; the constants are independent of dyadic numbers such as $\mu$ and $\lambda$;    $A\sim B$ means $B\lesssim A\lesssim B$;  $A\nsim B$ means either $A\ll  B$ or  $B\ll  A$;  $A\vee B$ and $A\wedge B$ denote the maximum and minimum of $A$ and $B$, respectively;
$\mathbbm{1}_{\{\cdot\}}$ denotes the indicator function which is 1 if the condition in the bracket is satisfied and 0 otherwise; we write  $a\pm:=a\pm \varepsilon$ for sufficiently small $0<\varepsilon\ll 1$. Finally, we use the notation
$$\| \cdot \|=\| \cdot \|_{L^2_{t,x}(\R^{1+3})} \quad \text{or} \quad \| \cdot \|_{L^2_{x}(\R^{3})}$$
depending on the context.

The Fourier transform in space and space-time are given by
\begin{align*}
\mathcal F_{x} (f)(\xi)=&\widehat f(\xi)=\int_{\R^3} e^{ -i x  \cdot \xi} f(x) \ dx,
\\
\mathcal F_{t, x} (u)(\tau,\xi)&=\widetilde u(\tau, \xi)=\int_{\R^{1+3}} e^{ -i(t\tau+ x \cdot \xi)} u(t,x) \ dt dx.
\end{align*}
Now
consider an even function $ \chi \in C_0^\infty((-2, 2))$ such that 
$\chi (s)=1 $ if $|s|\le 1$. We define 
\begin{align*}
\rho_{\lambda}(s)&=\begin{cases}
& 0, \quad \text{if} \ 0<\lambda <1,
\\
& \chi(s), \quad \text{if} \ \lambda=1,
\\
&\chi\left(\frac{s}{\lambda}\right)-\chi \left(\frac{2s}{\lambda}\right),\quad \text{if} \ \lambda>1
\end{cases}
 \end{align*}
 and
 \begin{align*}
 \sigma_{\lambda}(s)&=\chi\left(\frac{s}{\lambda}\right)-\chi \left(\frac{2s}{\lambda}\right) \ \text{for} \ \lambda>0.
 \end{align*}
Thus, $\supp \rho_{1}= \{ s\in \R: |s|<2 \}$ whereas $\supp \rho_{\lambda}= 
\{ s\in \R: \frac\lambda 2<|s|<2\lambda \}$ for $\lambda>1$. Similarly, 
 $\supp \sigma_{\lambda}= 
\{ s\in \R: \frac\lambda 2<|s|<2\lambda \}$ for all $\lambda>0$.
Then we define the frequency and modulation projections by
\begin{align*}
P_{\lambda} f &=\mathcal F_{x}^{-1} [\rho_\lambda(|\xi|)\widehat { f}(\xi) ],
\\
\Lambda^\pm_\lambda u&= \mathcal F_{t, x}^{-1}[\sigma_\lambda(|\tau\pm \angles{\xi}_m|)\widetilde{ u}(\tau, \xi)] .
 \end{align*}
Define also
\begin{align*}
\Lambda^\pm_{\ge \lambda} &=\sum_{\mu\ge \lambda }\Lambda^\pm_{\mu} ,
\quad
\Lambda^\pm_{< \lambda} =1-\Lambda^\pm_{\ge \lambda} .
 \end{align*}

\subsection{Function spaces: $U^p$ and $V^p$ spaces }
These function spaces
 were originally introduced 
in the unpublished work of Tataru on the wave map problem and 
then in Koch-Tataru \cite{KT07} in the context of NLS. The spaces
have since been used to obtain critical results in different problems 
related to dispersive equations (see eg. \cite{HHK09, HHK09-1,  HTT11}) and 
they serve as a useful replacement of $X^{s, b}$-spaces in the limiting cases. 
For the convenience of the reader we list the definitions and some properties of these spaces.

 Let $\mathcal{Z}$ be the collection of finite partitions 
 $-\infty < t_0 < \cdots < t_K \leq \infty$ of $\R$. If $t_K=\infty$, we use the convention $u(t_K) :=0$
 for all functions $u:\R\to L^2$.

 \begin{definition} \label{DefUp}
Let $1\leq p < \infty$.
A $U^p$-atom is defined by a step function $a:\R\to L^2$
of the form
$$
 a(t) = \sum_{k = 1}^K  \mathbb{1}_{[t_{k-1}, t_k) (t)} \phi_{k - 1}, 
 $$
where $$\{t_k\}_{k = 0}^K \in \mathcal{Z}, \quad 
\{\phi_k\}_{k = 0}^{K-1} \subset L^2 \ 
\text{with} \ \sum_{k = 0}^{K-1} \|\phi_k\|_{L^2}^p = 1.$$
The atomic space $U^p(\R; L^2)$
is defined to be the collection of functions $u:\R\to L^2$
of the form
\begin{equation} \label{Up}
u = \sum_{j = 1}^\infty \lambda_j a_j,
\ \text{ where $a_j$'s are $U^p$-atoms and} \ \{\lambda_j\}_{j \in \mathbb{N}}\in \ell^1,
\end{equation}
with the norm
$$\|u\|_{U^p} : = \inf_{\text{representation \eqref{Up} }}  \sum_{j=1}^{\infty} |\lambda_j|.
$$
\end{definition}

\begin{definition} \label{DefVp}
Let $1\leq p<\infty$. 

 \begin{enumerate}
\item \label{V:def} define $V^p(\R, L^2)$ as the space of all
  functions $v:\R\to L^2$ for which the norm
  \begin{equation}\label{Vp-norm}
    \|v\|_{V^p}:=\sup_{\{t_k\}_{k=0}^K \in \mathcal{Z}}
    \left(\sum_{k=1}^{K}
      \|v(t_{k})-v(t_{k-1})\|_{L^2}^p\right)^{\frac{1}{p}}
  \end{equation}
  is finite. 
  
 \item \label{V-:def} Likewise, let $V^p_-(\R, L^2)$ denote the normed space of all
  functions $v:\R\to L^2$ such that $\lim_{t\rightarrow -\infty} v(t)=0$ and $\|v\|_{V^p} < \infty$, endowed with the norm
\eqref{Vp-norm}.

\item \label{Vcr:def}  We let $V^p_{rc}(\R, L^2)$ 
  ($V^p_{-,rc} (\R, L^2)$) denote the closed subspace of all right-continuous 
  $V^p (\R, L^2)$ functions ($V^p_-(\R, L^2)$
  functions).

  \end{enumerate}

\end{definition}

\subsection{Properties of $U^p$ and $V^p$ spaces }
We collect some useful properties of these spaces. For more details about the spaces and
proofs we refer to \cite{HHK09, HHK09-1}.
\begin{proposition}\label{Prop-Up}
  Let $1\leq p< q < \infty$. Then we have the following:
  \begin{enumerate}
  \item \label{U-Banach} $U^p(\R, L^2)$ is a Banach space.
  \item\label{U-Emb} The embeddings $U^p(\R, L^2)\subset U^q(\R, L^2)\subset
    L^\infty(\R;L^2)$ are continuous.
  \item\label{U-RightCont} Every $u\in U^p(\R, L^2)$ is right-continuous. Moreover, $\lim_{t\to - \infty}u(t)=0$.
  
  \end{enumerate}
\end{proposition}
 
\begin{proposition}\label{Prop-Vp}Let $1\leq p<q <\infty$. Then we have the following:
  \begin{enumerate}
\item\label{V-Spaces} The spaces $V^p(\R, L^2)$, 
$V^p_{rc}(\R, L^2)$, $V^p_-(\R, L^2)$ and $V^p_{-,rc}(\R, L^2)$
  are Banach spaces.
\item \label{V-emb1} The embedding $U^p (\R, L^2) \subset V_{-,rc}^p (\R, L^2)$ is
  continuous.
\item \label{V-emb2} The embeddings $V^p (\R, L^2)\subset V^q (\R, L^2)$ and
  $V^p_-(\R, L^2) \subset V^q_- (\R, L^2)$ are continuous.
\item \label{V-embCont} The embedding $V^p_{-,rc}(\R, L^2) \subset U^q (\R, L^2)$ is continuous.
\end{enumerate}
\end{proposition}

\begin{lemma}[See \cite{KTV14}]\label{Lemma:Interp}
Let $p>2$ and $v\in V^2(\R, L^2)$. There exists $L=L(p)>0$ such that 
for all $N\ge 1$, there exist $w \in U^2(\R, L^2)$ and $z \in U^p(\R, L^2)$ with 
  $$ v=   w+   z$$ and 
\begin{equation*}
  \frac{L}N \norm { w}_{U^2}+ e^N
   \norm { z}_{U^p}\lesssim \norm { v}_{V^2}.
\end{equation*}
\end{lemma}

\begin{theorem}
\label{thm-dual}
  Let $1<p<\infty$. Then 
  $$ V^p=(U^p)^\ast $$
  in the sense that there is a bilinear form $B$ such that the mapping
  $$
  T: V^p\rightarrow (U^p)^\ast, \quad T(v):=B(\cdot, v)
  $$
  is an isometric isomorphism.
\end{theorem}

\begin{proposition}
\label{prop-dual}
  Let $1<p<\infty$ and $u \in V_-^1$
be absolutely continuous on compact
intervals and $v \in V^p$.
Then 
  $$
 B(u, v)= -\int_{-\infty}^\infty \angles{u'(t), v(t)} \, dt.
  $$
\end{proposition}

\subsection{ $U^p_\pm$ and $V^p_\pm$ spaces and their properties}

We now introduce $U^p,V^p$-type spaces that are adapted to the linear propagators $
S_m(\pm t)=e^{\mp it\angles{D}}.
$ 
\begin{definition}\label{UV-Propag}
We define $U^p_{\pm}(\R, L^2) $ (and $V^p_{\pm }  (\R, L^2)$, respectively)
to be the spaces of all functions $u: \R \mapsto L^2(\R^3)$
such that $t\to S_m(\mp t)u $ is in
$U^p (\R, L^2) $ (resp. $V^p(\R, L^2) $), with the respective norms:
\begin{align*}
 \|u \|_{U^p_{\pm}  } &=
 \| S_m(\mp t)u\|_{U^p},
\\
\|u \|_{V^p_{\pm}  } &= \| S_m(\mp t) u\|_{V^p}. 
\end{align*}
We use $V^p_{\text{rc}, \pm} (\R, L^2) $
to denote the subspace of right-continuous functions in
$V^p_{\pm} (\R, L^2)$.

 \end{definition}

 \begin{remark}
   Lemma \ref{Lemma:Interp}
   naturally extends to the spaces
   $U^p_{\pm }(\R, L^2) $ and $V^p_{\pm }  (\R, L^2)$.
 \end{remark}

\begin{lemma} [Interpolation] \label{lm-interpu2ub}
Let $p>2$ and $u_j:= P_{\lambda_j} u_j $ ($j=1, \cdots, 4$). For $u_j  \in U^2_{\epsilon_j}$ and  $u_4  \in V^2_{\epsilon_4}$, where $\epsilon_j \in \{+, -\}$,
define 
$$
I(\lambda):=\Big|\int   V\ast \innerprod{\beta u_1}{  u_2}   \cdot \innerprod{\beta  u_3}{  u_4}
  \ dt dx\Big|.
  $$
 Assume that the following estimate hold:
 \begin{equation}\label{u2ub}
  I(\lambda) \lesssim  \min \left( C_1( \lambda) \prod_{j=1}^3  \norm{  u_j }_{U^2_{\epsilon_j}}  \norm{  u_4 }_{U^2_{\epsilon_4}},
C_2( \lambda) \prod_{j=1}^3  \norm{ u_j}_{U^2_{\epsilon_j}}  \norm{u_4 }_{U^p_{\epsilon_4}} \right).
   \end{equation}
Then 
\begin{equation}\label{u2v2}
  I(\lambda) \lesssim C_1(\lambda)\left[ 1+ \ln ( C_2( \lambda)) \right]\prod_{j=1}^3  \norm{  u_j }_{U^2_{\epsilon_j}}  \norm{u_4 }_{V^2_{\epsilon_4}}.
   \end{equation}
  
\end{lemma}

\begin{proof}

Given $N\ge 1$, we use Lemma \ref{Lemma:Interp} 
to decompose $ u_4 \in V^2_{\epsilon_4}$ into 
   $ u_4=   u +   v$, where 
   $u \in U^2_{\epsilon_4}$ and $v \in U^p_{\epsilon_4}$,
   such that
\begin{equation}\label{interpest}
\left\{
\begin{aligned}
   \norm { u}_{U^2_{\epsilon_4}}&\lesssim \frac{N}{L} \norm { u_4}_{V^2_{\epsilon_4}},
   \\
   \norm { v}_{U^p_{\epsilon_4}}&\lesssim e^{-N}\norm {  u_4 }_{V^2_{\epsilon_4}}.
        \end{aligned}
\right.
   \end{equation}
We now use \eqref{u2ub} and \eqref{interpest} to obtain
  \begin{align*}
  I(\lambda)   & \lesssim C_1( \lambda) \prod_{j=1}^3  \norm{u_j}_{U^2_{\epsilon_j}}  \norm{u }_{U^2_{\epsilon_4}}
 +  C_2(\lambda) \prod_{j=1}^3  \norm{u_j }_{U^2_{\epsilon_j}}  \norm{v}_{U^p_{\epsilon_4}}
 \\
   &\lesssim \left[ \frac{N}{L} C_1(\lambda)  +   e^{-N}  C_2(\lambda)\right]\prod_{j=1}^3  \norm{u_j }_{U^2_{\epsilon_j}}  \norm{u_4 }_{V^2_{\epsilon_4}}.
   \end{align*}
This will imply the desired estimate \eqref{u2v2} if we choose
$$
N=1+ \ln [ C_2(\lambda) ].
$$ 
\end{proof}

\begin{lemma}[Modulation estimates; see \cite{HHK09}]\label{lm-modul}
Let $\lambda \in 2^k $, where $k\in \Z$, and $p\ge 2$. Then
\begin{align}
\label{modul1}
\| \Lambda^\pm_\lambda u \|_{L^2} &\lesssim \lambda^{-\frac12} \| u\|_{V^2_\pm}
\\
\label{modul2}
\| \Lambda^\pm_{\ge \lambda } u \|_{L^2} &\lesssim \lambda^{-\frac12} \| u\|_{V^2_\pm}
\\
\label{modul3}
\| \Lambda^\pm_{< \lambda } u \|_{V^p_\pm} &
\lesssim \| u\|_{V^p_\pm}, \quad 
\| \Lambda^\pm_{\ge \lambda } u \|_{V^p_\pm} 
\lesssim \| u\|_{V^p_\pm}
\\
\label{modul4}
\| \Lambda^\pm_{< \lambda } u \|_{U^p_\pm} &
\lesssim \| u\|_{U^p_\pm}, \quad 
\| \Lambda^\pm_{\ge \lambda } u \|_{U^p_\pm} 
\lesssim \| u\|_{U^p_\pm}.
\end{align}
\end{lemma}

\begin{lemma}\label{lm-transfer}
{\rm (Transfer principle)}
  Let $$T:L^2\times \cdots \times L^2\to L^1_{loc}(\R^3;\C)$$ be a
  multilinear operator and suppose that we have
\[ \big\| T(S_m(\pm t) \phi_1, \dots, S_m(\pm t) \phi_k)\big\|_{L^p_t L^r_x(\R\times \R^3)}
\lesssim \prod_{j = 1}^k \|\phi_j\|_{L^2_x ( \R^3)}\]
for some $1\leq p, r \leq \infty$.
Then
\[ \big\| T(u_1,  \dots, u_k)\big\|_{
L^p_t L^r_x(\R\times \R^3)}
\lesssim \prod_{j = 1}^k \|u_j\|_{U^p_{\pm} }.\]
\end{lemma}

%%%%%%%%%%%%%%%%%%%%%%%%%%%%%%%%%%%%%%%%%

%%%%%%%%%%%%%%%%%%%%%%%%%%%%%%%%%%%%%%%%%

\section{Linear and Bilinear estimates for free waves} \label{LinBi-Est}

 \subsection{Linear estimates}
The following Strichartz estimate for wave-admissible pairs is well-known.
\begin{lemma}[Wave-Strichartz]\label{lm-str}
Assume $m\ge 0$, $2\le r <\infty $ and $\frac1q + \frac1r=\frac12$. Then
$$
\norm{S_m(\pm t) f_\lambda}_{L^q_t L^r_x }\lesssim \lambda^{\frac{2}{q}} \norm{  f_\lambda }_{L^2}.
 $$
 Moreover, for all $u_\lambda \in U^q_\pm$, we have (by the transfer-principle) 
 $$
\norm{u_\lambda}_{L^q_t L^r_x }\lesssim \lambda^{\frac{2}{q}} \norm{  u_\lambda }_{U^q_\pm}.
 $$
\end{lemma}

\subsection{Bilinear estimates}

The following Lemma contains estimates for convolution of two free waves. This generalizes the result of Foschi–Klainerman for $m = 0$ 
\cite[Lemma 4.1 and Lemma 4.4]{FK00} to $m\ge 0$. 
 The proof is given in Section \ref{seclmI+-}.
\begin{lemma}[Convolution of free waves]\label{lmI+-}
For $m\ge 0$ define
\begin{align*}
I_+(f,g)(\tau, \xi)&= \int_{\R^3} f(|\eta|) g(|\xi-\eta|) \delta( \tau-\angles{\eta}_m- \angles{\xi-\eta}_m) \, d\eta,
\\
I_-(f, g)(\tau, \xi)&= \int_{\R^3} f(|\eta|) g(|\xi-\eta|) \delta( \tau-\angles{\eta}_m+ \angles{\xi-\eta}_m) \, d\eta.
\end{align*}
 Then the following hold:
 \begin{enumerate}[(i).]
 \item \label{lmI+-i} Estimate for $I_+$:
\begin{equation} \label{I+}
I_+(f, g)(\tau, \xi)
\simeq \frac{1}{|\xi|}\int_{a_- }^{a_+ } 
r(\tau-r) f\left( \sqrt{r^2-m^2}\right)  g\left( \sqrt{ (\tau-r)^2-m^2} \right)  
  dr,
  \end{equation}
where
 $$
 a_\pm:=a_\pm(\tau, \xi)=\frac\tau2 \pm \frac{|\xi|}2 \sqrt{\frac{\tau^2-|\xi|^2-4m^2}{\tau^2-|\xi|^2}}.
 $$
 \item \label{lmI+-ii} Estimate for $I_-$:
 \begin{equation} \label{I-}
\begin{split}
 I_-(f, g)(\tau, \xi)
& \simeq \frac1{|\xi|}\int_{a_+}^{\infty}
r(r-\tau) f\left( \sqrt{r^2-m^2}\right) g\left( \sqrt{ (r-\tau)^2-m^2} \right)
  dr.
  \end{split}
 \end{equation}
 
  \end{enumerate}

\end{lemma}

Lemma \ref{lmI+-} is used to prove the key bilinear estimates in Lemma \ref{lmBiest} below,  
 which also generalizes the result of Foschi–Klainerman for $m = 0$ 
\cite[Lemma 12.1]{FK00} to $m\ge 0$. 
 The proof is given in the Section \ref{seclmBiest}.
  
\begin{lemma}[Bilinear estimates for free waves]\label{lmBiest}
Let $m\ge 0$ and $\mu, \lambda_1, \lambda_2 \ge 1$.  Then for all $f_{ \lambda_1}, \ g_{\lambda_2 } \ \in L^2_x$ we have the following:
\begin{enumerate}[(i)]
\item (++) interaction: \label{Biest++}
\begin{align*}
\norm{   P_\mu \left ( S_m(t) f_{ \lambda_1}  \cdot S_m(t) g_{\lambda_2 } \right)  }
\lesssim
\begin{cases}
& (\lambda_1 \wedge \lambda_2) \norm{f_{ \lambda_1}} \norm{g_{ \lambda_2}}\quad  \text{if} \  \ \lambda_1\nsim \lambda_2,
\\
& \mu \norm{f_{ \lambda_1}} \norm{g_{ \lambda_2}}\quad  \text{if} \ \  \mu \lesssim \lambda_1 \sim\lambda_2.
\end{cases}
\end{align*}

\item (+-) interaction:  \label{Biest+-}
\begin{align*}
\norm{   P_\mu \left ( S_m(t) f_{ \lambda_1}  \cdot S_m(-t) g_{\lambda_2 } \right)  }
\lesssim
\begin{cases}
& (\lambda_1 \wedge \lambda_2) \norm{f_{ \lambda_1}} \norm{g_{ \lambda_2}} \quad  \text{if} \ \ \lambda_1\nsim \lambda_2,
\\
& \mu^\frac12 \lambda_1^\frac12  \norm{f_{ \lambda_1}} \norm{g_{ \lambda_2}} \quad  \text{if} \ \   \mu \lesssim \lambda_1 \sim\lambda_2.
\end{cases}
\end{align*}

\end{enumerate}

\end{lemma}

%%%%%%%%%%%%%%%%%%%%%%%%%%%%%%%%%%%%%%%%%
\section{Null structure, Null-form and bilinear estimates} \label{secnull}
\subsection{Null structure } 
Here we reveal the null structure in the bilinear terms $\innerprod{\beta\psi^+}{\psi ^\pm}$.
 Taking Fourier Transform in space and then using the identities \eqref{psi-split}--\eqref{Pbetaidentities} we obtain
\begin{equation}\label{PFT}
\begin{split}
\mathcal F_{x} \innerprod{\beta \Pi^+_m(D)\psi^+ }{\Pi^\pm_m(D) \psi^\pm}( \xi) 
 &=\iint_{\xi=\eta-\zeta}\innerprod{\beta \Pi^+_m(\eta) \widehat {\psi^+} ( \eta) }{\Pi^\pm_m (\zeta)  \widehat {\psi^\pm} (\zeta)} \, d\eta  d\zeta 
   \\
   &=\iint_{\xi=\eta-\zeta} \innerprod{ \beta \widehat { \psi^+}  (\eta) }{ \Pi^-_m(\eta)\Pi^\pm_m(\zeta)   \widehat {\psi^\pm} (\zeta)} \, d\eta  d\zeta 
   \\
   & \qquad \qquad +m \iint_{\xi=\eta-\zeta} \angles{\eta}_m^{-1}
\innerprod{ \widehat { \psi^+}  ( \eta) }{  \widehat {\psi^\pm} (\zeta)}\, d\eta  d\zeta .
\end{split}
\end{equation}

We use the notation 
$\hat{\xi}=\angles{\xi}_m^{-1}\xi$. Now
we compute 
\footnote{In view of \eqref{matrixidentities2} we have 
\begin{align*}
 I\mp (\boldsymbol  \alpha \cdot\hat{\eta} )(\boldsymbol  \alpha \cdot \hat{\zeta} ) &= I\mp \boldsymbol  \alpha^j  \boldsymbol  \alpha^k \hat{\eta}_j \hat{\zeta}_k 
 = I\mp \delta^{jk} \hat{\eta}_j \hat{\zeta}_k I + i \hat{\eta}_j \hat{\zeta}_k \epsilon^{jkl} S_{l}
\\
&= (1 \mp \hat{\eta}\cdot\hat{\zeta}) I \mp i (\hat{\eta} \times \hat{\zeta} )\cdot S
\end{align*}
},
\footnote{If $m=0$ then $\hat{\eta}=|\eta|^{-1}\eta$, $\hat{\zeta}=|\zeta|^{-1}\zeta$  and $r^\pm(\eta, \zeta)=0$. Hence 
$$4 \Pi^-_0(\eta)\Pi^\pm_0 (\zeta) = \left(1 \mp (|\eta||\zeta|)^{-1} \eta\cdot\zeta \right) I \mp i (|\eta||\zeta|)^{-1} (\eta \times \zeta )\cdot S - \left( |\eta|^{-1}\eta\mp |\zeta|^{-1} \zeta\right) \cdot  \boldsymbol  \alpha$$
which coincides with the null structure found in \cite[See Lemma 2 therein]{DFS06}}.
\begin{align*}
 4 \Pi^-_m(\eta)\Pi^\pm_m (\zeta)&= \left( I -\frac{1}{\angles{\eta}_m}[ \boldsymbol \alpha \cdot \eta  + m\beta] \right) 
 \left( I \pm \frac{1}{\angles{\zeta}_m}[ \boldsymbol \alpha \cdot\zeta + m \beta]\right)
 \\
 &= I\mp (\boldsymbol  \alpha \cdot\hat{\eta} )(\boldsymbol  \alpha \cdot \hat{\zeta} ) - (\hat{\eta}\mp \hat{\zeta}) \cdot \boldsymbol  \alpha + r^\pm(\eta, \zeta)
  \\
 &= (1 \mp \hat{\eta}\cdot\hat{\zeta}) I \mp i (\hat{\eta} \times \hat{\zeta} )\cdot S - ( \hat{\eta}\mp \hat{\zeta}) \cdot  \boldsymbol  \alpha + r^\pm  (\eta, \zeta),
\end{align*}
where
\begin{align*}
r^\pm(\eta, \zeta)&= \mp 
  \frac{ [(\eta-\zeta)\cdot  \boldsymbol  \alpha] m\beta -(\angles{\eta}\mp\angles{\zeta} )m\beta  + m^2 I}{\angles{\eta}_m\angles{\zeta}_m}.
\end{align*}

We can write
\begin{align*}
(1 \mp \hat{\eta}\cdot\hat{\zeta}) I &= q^\pm_1(\eta, \zeta)+b^\pm_1(\eta, \zeta),
\\
- (\hat{\eta} \mp\hat{\zeta})\cdot \boldsymbol  \alpha &= q^\pm_2(\eta, \zeta)+b^\pm_2(\eta, \zeta),
\end{align*}
where
\begin{equation}\label{qb12}
\left\{
\begin{aligned}
q^\pm_1(\eta, \zeta)&=( |\hat{\eta}| | \hat{\zeta}|\mp \hat{\eta}\cdot\hat{\zeta})I,  \quad
q^\pm_2(\eta, \zeta)= -(\hat \eta| \hat \zeta| \mp \hat \zeta|\hat \eta| )\cdot \boldsymbol \alpha,  
\\
b^\pm_1(\eta, \zeta)&= (1 -|\hat{\eta}| | \hat{\zeta}| ) I,   \quad b^\pm_2(\eta, \zeta)=- \frac{ [\eta( \angles{\zeta}_m- |\zeta|)\mp \zeta( \angles{\eta}_m- |\eta|)]\cdot \boldsymbol  \alpha}{\angles{\eta}_m\angles{\zeta}_m}.
\end{aligned}
\right.
\end{equation}
Setting
\begin{equation}\label{qb3}
q^\pm_3(\eta, \zeta):= \mp i (\hat{\eta} \times \hat{\zeta} )\cdot S, \quad 
b^\pm_3(\eta, \zeta):=r^\pm(\eta, \zeta)
\end{equation}
we can write
\begin{equation}\label{Psymb}
 4 \Pi^-_m(\eta)\Pi^\pm_m(\zeta)=\sum_{j=1}^3 [q^\pm_j(\eta, \zeta) +b^\pm_j(\eta, \zeta)],
\end{equation}
where $q^\pm_j$ and $b^\pm_j$ for $j=1,2,3$ are given above.

Then in view of \eqref{PFT}--\eqref{Psymb} we can write
\begin{equation}\label{pp}
 \innerprod{\beta \Pi^+_m(D)\psi^+ }{\Pi^\pm_m(D) \psi^\pm}
   =\sum_{j=1}^3  Q_j(\psi^+, \psi^\pm)+ \sum_{j=1}^4  B_j(\psi^+, \psi^\pm),
\end{equation}
where
\begin{equation}\label{qb}
\left\{
\begin{aligned}
 Q_j(\psi^+, \psi^\pm)&= \mathcal F_x^{-1}\iint_{\xi=\eta-\zeta} \innerprod{   \beta \widehat \psi^+  (\eta) }{ q_j^\pm(\eta, \zeta)\widehat \psi^\pm (\zeta)} \, d\eta  d\zeta ,
   \\
  B_j(\psi^+, \psi^\pm)&=\mathcal F_x^{-1} \iint_{\xi=\eta-\zeta}
    \innerprod{  \beta \widehat \psi^+  (\eta) }{ b_j^\pm(\eta, \zeta)\widehat \psi^\pm (\zeta)}\, d\eta  d\zeta 
\end{aligned}
\right.
\end{equation}
with 
\begin{equation}\label{b4}
 b_4^\pm(\eta, \zeta)= m\angles{\eta}_m^{-1} \beta. 
\end{equation}

As we show below the $Q_j$'s are all null
forms. The $B_j$'s on the other hand, are not. Nevertheless, these 
terms are more regular than the $Q_j$'s with one full derivative, see the estimates in Lemma \ref{lm-biest-strz} below. Thus, we have decomposed $ \innerprod{\beta \Pi^+_m(D)\psi^+ }{\Pi^\pm_m(D) \psi^\pm}$ into sum of bilinear null forms and generic bilinear terms that are smoother. It is worth to note that if $m=0$ then all the $B_j$'s are zero (see the footnote above),and hence we only have the null forms on the r.h.s of \eqref{pp}.

The null symbols $q^\pm_j$ satisfy the following estimates.
\begin{lemma}\label{lm-qest}
Let $a\in \left[0, \frac12 \right]$. Then for $j=1, 2,3$ we have 
\begin{equation}\label{q-est}
\left\{
\begin{aligned}
|q^+_j(\eta, \zeta)| &\lesssim \left[ \frac{    |\eta-\zeta|(  |\eta-\zeta|- ||\eta|-|\zeta||)  }{ \angles{\eta}_m \angles{\zeta}_m }\right]^{a},
\\
|q^-_j(\eta, \zeta)| &\lesssim \left[ \frac{   ( |\eta| +|\zeta|)(  |\eta| + |\zeta|- |\eta-\zeta|)  }{ \angles{\eta}_m \angles{\zeta}_m }\right]^{a}.
\end{aligned}
\right.
\end{equation}

\end{lemma}
\begin{proof}
First note that
\begin{equation}
\label{qle1}
|q_j^\pm(\eta, \zeta)| \lesssim 1.
\end{equation}
Moreover, for two vectors $\eta$ and $\zeta$ we have the following estimates
 (see e.g. \cite[Lemma 13.2]{FK00}):
 \begin{equation}\label{qj}
\left\{
\begin{aligned}
|q^+_1(\eta, \zeta)|  & \sim \frac{ |\eta -\zeta|  \left( |\eta -\zeta| -| |\eta|-|\zeta| | \right)}{\angles{\eta}_m\angles{\zeta}_m},
\\
|q^-_1(\eta, \zeta)| & \sim \frac{ (  |\eta| + |\zeta|) \left(   |\eta| + |\zeta|- |\eta-\zeta| \right)}{\angles{\eta}_m\angles{\zeta}_m} ,
\\
|q^+_{j}(\eta, \zeta)| &\lesssim  \left[\frac{  |\eta-\zeta| \left(  |\eta-\zeta| -| |\eta|-|\zeta| |  \right)}{ \angles{\eta}_m\angles{\zeta}_m} \right]^\frac12,
\\
|q^-_{j}(\eta,\zeta)| & \lesssim \left[ \frac{ ( |\eta|+|\zeta|) \left(   |\eta| + |\zeta|- |\eta-\zeta| \right)}{\angles{\eta}_m\angles{\zeta}_m}\right]^\frac12,
\end{aligned}
\right.
\end{equation}
where $j=2, 3$.
 Now interpolation between \eqref{qle1} and \eqref{qj} gives the desired estimates in 
\eqref{q-est}.
\end{proof}

\subsection{Null-form estimates}
We now prove bilinear null-form estimates for two free solutions 
$S_m(t) f$ and $S_m(\pm t)  g$ of the Dirac equation.

\begin{lemma}[Null-form estimates for free waves]\label{lmNullest}
Let $m\ge 0$ and $\mu, \lambda_1, \lambda_2 \ge 1$. Then
we have the following null-form estimates:
\begin{enumerate}[(i)]
\item (++) interaction: \label{Nullest++}
\begin{align*}
\norm{ P_\mu Q_j( S_m(t)f_{\lambda_1},  S_m(t)g_{ \lambda_2})}
\lesssim
\begin{cases}
& (\lambda_1 \wedge \lambda_2) \norm{f_{ \lambda_1}} \norm{g_{ \lambda_2}}\quad  \text{if} \ \lambda_1\nsim \lambda_2,
\\
& \mu (\mu /\lambda_1 )^{\frac12}   \norm{f_{ \lambda_1}} \norm{g_{ \lambda_2}}\quad  \text{if} \ \  \mu \lesssim \lambda_1 \sim\lambda_2.
\end{cases}
\end{align*}

\item (+-) interaction:  \label{Nullest+-}
\begin{align*}
\norm{ P_\mu Q_j( S_m(t)f_{\lambda_1},  S_m(-t)g_{ \lambda_2})}
\lesssim
\begin{cases}
& (\lambda_1 \wedge \lambda_2) \norm{f_{ \lambda_1}} \norm{g_{ \lambda_2}} \quad  \text{if} \ \ \lambda_1\nsim \lambda_2,
\\
& \mu  \norm{f_{ \lambda_1}} \norm{g_{ \lambda_2}} \quad  \text{if} \ \   \mu \lesssim \lambda_1 \sim\lambda_2.
\end{cases}
\end{align*}

\end{enumerate}

\end{lemma}
\begin{proof}

By \eqref{qb} we have
\begin{align*}
&\Bigabs{ \mathcal F_{t,x} [Q_j( S_m(t)f_{\lambda_1},  S_m(\pm t)g_{ \lambda_2})](\tau, \xi)}
\\
& \qquad =\Bigabs{\iint_{\xi=\eta-\zeta}  \rho_\mu(|\eta-\zeta|) 
 \innerprod{   \beta \widehat { f_{\lambda_1}} (\eta) }{ q_j^\pm (\eta, \zeta)\widehat { g_{\lambda_2}}  (\zeta)}  \delta(\tau+\angles{\eta}_m \mp \angles{\zeta}_m)  \, d\eta  d\zeta  }
 \\
& \qquad \lesssim \iint_{\xi=\eta-\zeta}  \rho_\mu(|\eta-\zeta|)  | q_j^\pm (\eta, \zeta)| 
  |\widehat{ f_{\lambda_1} }(\eta)| |   \widehat{ g_{\lambda_2} }(\zeta)|  \delta(\tau+\angles{\eta}_m\mp \angles{\zeta}_m) \, d\eta  d\zeta ,
\end{align*}
where on the second line the sign change to $\mp$ in the delta function is because of the complex conjugation in $\innerprod{\cdot}{\cdot}$.

By \eqref{q-est} we have
\begin{equation}\label{qj-dydicest}
q^+_j(\eta, \zeta)| \lesssim  \left(\frac{   \mu( \mu \wedge \lambda_1 \wedge \lambda_2) }{ \lambda_1\lambda_2 }\right)^{\frac12}  \quad \text{and} \quad |q^-_j(\eta, \zeta)| \lesssim 1,
\end{equation}
where for the estimate on $q_j^+$ we used $$  |\eta-\zeta|- ||\eta| - |\zeta|| \lesssim   \mu \wedge \lambda_1 \wedge \lambda_2.$$ 
Indeed, this is obvious if $\mu \lesssim \lambda_1\sim \lambda_2$. Now assume
 $ \lambda_2 \ll \lambda_1\sim \mu$.  Then
 \begin{align*}
 |\eta-\zeta|- ||\eta| - |\zeta||=\frac{ |\eta-\zeta|^2- (|\eta| - |\zeta|)^2 }{|\eta-\zeta|+ ||\eta| - |\zeta||}= \frac{ 2|\eta||\zeta| -2\eta\cdot \zeta}{|\eta-\zeta|+ ||\eta| - |\zeta||}  \lesssim \lambda_2.
 \end{align*}
 The case  $ \lambda_1 \ll \lambda_2\sim \mu$ also follows by symmetry.

Now using the estimate for  $q_j^+$ in \eqref{qj-dydicest} we have for the (++) interaction 
\begin{align*}
& \Bigabs{\mathcal F_{t,x} [P_\mu Q_j( S_m(t)f_{\lambda_1},  S_m(t)g_{ \lambda_2})](\tau, \xi)}
  \\
&\qquad \lesssim \left(\frac{   \mu( \mu \wedge \lambda_1 \wedge \lambda_2) }{ \lambda_1\lambda_2 }\right)^{\frac12}   \iint_{\xi=\eta-\zeta} \rho_\mu(|\eta-\zeta|)  
  |\widehat{ f_{\lambda_1}}(\eta)| |   \widehat { g_{\lambda_2}}(\zeta)| \delta(\tau+\angles{\eta}_m-\angles{\zeta}_m) \, d\eta  d\zeta 
  \\
  &\qquad = \left(\frac{   \mu( \mu \wedge \lambda_1 \wedge \lambda_2) }{ \lambda_1\lambda_2 }\right)^{\frac12} \mathcal F_{t,x} \left [P_\mu \left( S(t) \mathcal F_x^{-1}(|\widehat{ f_{\lambda_1}} | ) \cdot 
   S(-t) \mathcal F_x^{-1}(|\widehat{ g_{\lambda_2}} |  ) \right) \right](\tau, \xi),
\end{align*}
By Plancherel and Lemma \ref{lmBiest}\eqref{Biest+-} we obtain
\begin{align*}
\norm{ P_\mu Q_j( S_m(t)f_{\lambda_1},  S_m(t)g_{\lambda_2})}
&\lesssim \left(\frac{   \mu( \mu \wedge \lambda_1 \wedge \lambda_2) }{ \lambda_1\lambda_2 }\right)^{\frac12}\norm{  P_\mu \left( S_m(t) \mathcal F_x^{-1}(|\widehat{ f_{\lambda_1}} | ) \cdot 
   S_m(-t) \mathcal F_x^{-1}(|\widehat{ g_{\lambda_2}} |  ) \right)} 
\\
& \lesssim 
\begin{cases}
& (\lambda_1 \wedge \lambda_2) \norm{f_{ \lambda_1}} \norm{g_{ \lambda_2}} \quad  \text{if} \ \ \lambda_1\nsim \lambda_2,
\\
& \mu (\mu /\lambda_1 )^{\frac12}  \norm{f_{ \lambda_1}} \norm{g_{ \lambda_2}} \quad  \text{if} \ \ \mu \lesssim \lambda_1 \sim\lambda_2.
\end{cases}
\end{align*}

Similarly, we use the estimate for $q_j^-$ in \eqref{qj-dydicest} to estimate the (+-) interaction as
\begin{align*}
&\Bigabs{ \mathcal F_{t,x} [ P_\mu Q_j( S(t)f_{\lambda_1},  S(-t)g_{\lambda_2})](\tau, \xi)}
\\
&\qquad \lesssim \iint_{\xi=\eta-\zeta}  \rho_\mu(|\eta-\zeta|)  | q_j^-(\eta, \zeta)| 
  |\widehat{ f_{\lambda_1}}(\eta)| |   \widehat { g_{\lambda_2}}(\zeta)|  \delta(\tau+\angles{\eta}_m+\angles{\zeta}_m)  \, d\eta
  \\
  &\qquad =  \mathcal F_{t,x} \left [ P_\mu \left( S_m(t) \mathcal F_x^{-1}(|\widehat{ f_{\lambda_1}} | ) \cdot 
   S_m(t) \mathcal F_x^{-1}(|\widehat{ g_{\lambda_2}} |  ) \right) \right](\tau, \xi).
\end{align*}
By Plancherel and Lemma \ref{lmBiest}\eqref{Biest++}  we obtain
\begin{align*}
\norm{ P_\mu Q_j( S_m(t)f_{\lambda_1},  S_m(-t)g_{\lambda_2})} &\lesssim  
\norm{ P_\mu \left( S_m(t) \mathcal F_x^{-1}(|\widehat{ f_{\lambda_1}} | ) \cdot 
   S_m(t) \mathcal F_x^{-1}(|\widehat{ g_{\lambda_2}} |  ) \right)} 
\\
&\lesssim \begin{cases}& (\lambda_1 \wedge \lambda_2)  \norm{f_{\lambda_1}} \norm{g_{\lambda_2}}  \quad \text{if} \ \   \lambda_1\nsim\lambda_2
 \\
 & \mu  \norm{f_{\lambda_1}} \norm{g_{\lambda_2}} \quad \text{if} \    \mu\lesssim \lambda_1 \sim\lambda_2 .
 \end{cases}
\end{align*}

\end{proof}

Applying Lemma \ref{lm-transfer}  (the transfer principle) to Lemma \ref{lmNullest} we obtain the following.
\begin{corollary}[Null-form estimates in the $U^2$-space]\label{lmNullestU2}
Let $m\ge 0$ and $\mu, \lambda_1, \lambda_2 \ge 1$.
\begin{enumerate}[(i)]
\item (++) interaction: \label{NullestU2++} For all $u_{ \lambda_1} , v_{ \lambda_2} \in U^2_+$ we have
\begin{align*}
\norm{ P_\mu Q_j(u_{ \lambda_1 }, v_{ \lambda_2})} \lesssim 
\begin{cases}
 & (\lambda_1\wedge \lambda_2) \norm{u_{\lambda_1} }_{U^2_+} \norm{v_{\lambda_2} }_{U^2_+}  \quad \text{if} \ \   \lambda_1\nsim\lambda_2 ,
 \\
 &  \mu(\mu/\lambda_1)^\frac12  \norm{u_{\lambda_1} }_{U^2_+} \norm{v_{\lambda_2} }_{U^2_+}   \quad \text{if} \ \    \mu \lesssim \lambda_1 \sim\lambda_2.
 \end{cases}
\end{align*}

\item (+-) interaction:  \label{NullestU2+-}  For all $u_{\lambda_1}\in U^2_+$ and $v_{ \lambda_2} \in U^2_-$, we have
\begin{align*}
\norm{ P_\mu Q_j(u_{ \lambda_1 }, v_{ \lambda_2})} \lesssim 
\begin{cases}
 & (\lambda_1\wedge \lambda_2) \norm{u_{\lambda_1} }_{U^2_+} \norm{v_{\lambda_2} }_{U^2_-}  \quad \text{if} \ \   \lambda_1\nsim\lambda_2 ,
 \\
 &  \mu  \norm{u_{\lambda_1} }_{U^2_+} \norm{v_{\lambda_2} }_{U^2_-}   \quad \text{if} \ \   \mu \lesssim \lambda_1 \sim\lambda_2.
 \end{cases}
\end{align*}

\end{enumerate}

\end{corollary}

\subsection{Bilinear estimates} \label{SecBiest}
In this section we express the bilinear terms in \eqref{qb},  $Q_j$ and $B_j$, in physical space. We then apply Cauchy-Schwarz and Strichartz estimates to derive bilinear estimates for $Q_j$ and $B_j$.

\begin{lemma}\label{lm-biest-strz}
Let $Q$ denote any one of the
$Q_j$'s ($j=1, \cdots, 3$) and $B$ denote any one of the $B_j$'s ($j=1, \cdots, 4$). For $\lambda_1, \lambda_2\ge 1$ assume that $u_{ \lambda_1} \in V^2_{\pm}$  and $v_{ \lambda_2} \in V^2_{\pm'}$, where $\pm$ and $\pm'$ are two independent signs. Then
\begin{align}
\label{BiestQU4}
    \norm{  P_\mu  Q(u_{ \lambda_1}, v_{ \lambda_2} )  } 
&\lesssim  (\lambda_1 \lambda_2)^{\frac12}  \norm{u_{ \lambda_1}}_{U^4_{\pm}}  \norm{v_{ \lambda_2}}_{U^4_{\pm'}} ,
\\
\label{BiestQV2}
    \norm{  P_\mu  Q(u_{ \lambda_1}, v_{ \lambda_2} )  } 
&\lesssim  (\lambda_1 \lambda_2)^{\frac12}  \norm{u_{ \lambda_1}}_{V^2_{\pm}}  \norm{v_{ \lambda_2}}_{V^2_{\pm'}} ,
\\
\label{BiestB}
    \norm{  P_\mu  B(u_{ \lambda_1}, v_{ \lambda_2} )  } 
&\lesssim \frac{(\lambda_1 \lambda_2)^{\frac12}}{ (\lambda_1\wedge \lambda_2)}  \norm{u_{ \lambda_1}}_{V^2_{\pm}}  \norm{v_{ \lambda_2}}_{V^2_{\pm'}} .
   \end{align}

\end{lemma}
\begin{proof}
By Proposition \ref{Prop-Vp}\eqref{V-embCont}, the estimate \eqref{BiestQV2} follows from \eqref{BiestQU4}. Thus, we only need to prove \eqref{BiestQU4} and \eqref{BiestB}.

The null forms $Q_j( u, v)$ in \eqref{qb} can be written in physical space as follows:
\begin{equation}\label{Nullform-phy}
\begin{split}
  Q_1(u, v)&= \innerprod{ \beta R u   }{R v }\mp \innerprod{\beta  R_j  u   }{R^j v },
  \\
    Q_2(u, v)&= \innerprod{\beta  R_j u   }{ \alpha^j R  v }\pm \innerprod{  \beta R u   }{ \alpha^jR_j v },
  \\
   Q_3(u, v)&= \innerprod{\beta  R_1 u   }{ \gamma  R_2 v }\pm \innerprod{  \beta R_2 u   }{ \gamma R_1 v },
   \end{split}
\end{equation}
where $$ R_j= \frac{ \partial_j}{\angles{D}_m}\quad \text{and} \quad    R= \frac{ |D| }{\angles{D}_m} $$ are Riesz operators.
Theses operators are bounded in $L^p$ for $1<p<\infty$, i.e., 
\begin{equation}\label{Reiz-R}
\norm{  R_j f_\lambda}_{L^p}\lesssim \norm{ f_\lambda }_{L^p}\quad \text{and} \quad   \norm{  R f_\lambda}_{L^p} \lesssim \norm{ f_\lambda }_{L^p}.
\end{equation}
Now by H\"{o}lder, \eqref{Reiz-R} and Lemma \ref{lm-str} we have 
\begin{align*}
    \norm{  P_\mu  Q_j(u_{ \lambda_1}, v_{ \lambda_2} )  } 
&\lesssim  \norm{u_{ \lambda_1}}_{L^4_{t,x}}  \norm{v_{ \lambda_2}}_{L^4_{t,x}}
\\
& \lesssim (\lambda_1 \lambda_2)^{\frac12}  \norm{u_{ \lambda_1}}_{U^4_{\pm}}  \norm{v_{ \lambda_2}}_{U^4_{\pm'}} .
   \end{align*}

Next we prove \eqref{BiestB}. The bilinear terms $B_j( u, v)$ in \eqref{qb} can be written in physical space as follows:
\begin{equation}\label{Biform-phy}
\begin{split}
    B_1(u, v)&=  \innerprod{ \beta (1-R)u   }{ v }+\innerprod{\beta Ru   }{(1-R) v },
    \\
    B_2(u, v)&= - \innerprod{ \beta R_j u   }{ \alpha^j (1-R)v }\pm 
    \innerprod{ (1-R)\beta u   }{\alpha^j R_j  v },
    \\
    B_3(u, v)&= \mp \innerprod{ u   }{   \angles{D}_m^{-1} v }+
    \innerprod{  \angles{D}_m^{-1} u   }{ v } \mp  \innerprod{  R_j  u   }{\alpha^j \angles{D}^{-1}v }
    \\
   &  \qquad  \mp  \innerprod{  \angles{D}_m^{-1}   u   }{\alpha^j R_jv }\mp  \innerprod{  \beta \angles{D}_m^{-1}   u   }{\angles{D}_m^{-1}  v },
    \\
  B_4(u, v)&=- \innerprod{   \angles{D}_m^{-1}u   }{ v }.
\end{split}
\end{equation}
Note that
\begin{equation}\label{Reiz-D}
\norm{   \angles{D}_m^{-1}  f_\lambda}_{L^2}\lesssim   \angles{\lambda}_m^{-1} \norm{ f_\lambda }_{L^2}
\end{equation}
and 
\begin{equation}\label{Reiz-1R}
\norm{(1-R)f_\lambda}_{L^2} \lesssim \angles{\lambda}_m^{-2}\norm{f_\lambda}_{L^2},
\end{equation}
where in the latter case we used Plancherel and the fact that
\begin{align*}
1-\frac{ |\xi|}{\angles{\xi}_m}=\frac{\angles{\xi}_m-|\xi|}{\angles{\xi}_m}=\frac{m^2}{\angles{\xi}_m (\angles{\xi}_m+|\xi|)}\sim m^2\angles{\xi}_m^{-2}.
\end{align*}

Now using H\"{o}lder,  Lemma \ref{lm-str}, Proposition \ref{Prop-Vp}\eqref{V-embCont} and \eqref{Reiz-R}--\eqref{Reiz-1R} we obtain
\begin{align*}
    \norm{  P_\mu  B_1(u_{ \lambda_1}, v_{ \lambda_2} )  } 
&\lesssim  \norm{(1-R)u_{ \lambda_1}}_{L^4_{t,x}}  \norm{v_{ \lambda_2}}_{L^4_{t,x}}  + \norm{u_{ \lambda_1}}_{L^4_{t,x}}  \norm{ (1-R)v_{ \lambda_2}}_{L^4_{t,x}} 
\\
&\lesssim (\lambda_1 \lambda_2)^{\frac12}  \left\{\norm{(1-R)u_{ \lambda_1}}_{U^4_{\pm}}  \norm{v_{ \lambda_2}}_{U^4_{\pm'}} +  \norm{u_{ \lambda_1}}_{U^4_{\pm}}  \norm{ (1-R)v_{ \lambda_2}}_{U^4_{\pm'}} \right\}.
\\
&\lesssim (\lambda_1 \lambda_2)^{\frac12}  \left\{\norm{(1-R)u_{ \lambda_1}}_{V^2_{\pm}}  \norm{v_{ \lambda_2}}_{V^2_{\pm'}} +  \norm{u_{ \lambda_1}}_{V^2_{\pm}}  \norm{ (1-R)v_{ \lambda_2}}_{V^2_{\pm'}} \right\}.
\\
&\lesssim \frac{(\lambda_1 \lambda_2)^{\frac12}}{ (\lambda_1\wedge \lambda_2)^2} \norm{u_{ \lambda_1}}_{V^2_{\pm}}  \norm{v_{ \lambda_2}}_{V^2_{\pm'}} .
   \end{align*}
Similarly, 
\begin{align*}
    \norm{  P_\mu  B_2(u_{ \lambda_1}, v_{ \lambda_2} )  } 
&\lesssim \frac{(\lambda_1 \lambda_2)^{\frac12}}{ (\lambda_1\wedge \lambda_2)^2}  \norm{u_{ \lambda_1}}_{V^2_{\pm}}  \norm{v_{ \lambda_2}}_{V^2_{\pm'}} 
   \end{align*}
and for $j=3, 4$
\begin{align*}
    \norm{  P_\mu  B_j(u_{ \lambda_1}, v_{ \lambda_2} )  } 
&\lesssim \frac{(\lambda_1 \lambda_2)^{\frac12}}{ (\lambda_1\wedge \lambda_2)}   \norm{u_{ \lambda_1}}_{V^2_{\pm}}  \norm{v_{ \lambda_2}}_{V^2_{\pm'}} 
   \end{align*}

\end{proof}

%%%%%%%%%%%%%%%%%%%%%%%%%%%%%%%%%%%%%%%%%

%%%%%%%%%%%%%%%%%%%%%%%%%%%%%%%%%%%%%%%%%

 \section{Reduction of Theorem \ref{mainthm1} to nonlinear estimates}
\label{sec-ProofMainThm}

 Let $I=[0, \infty)$.
We define $X^s_\pm $ to be the complete space of all functions 
 $u:I \to L^2$ such that $P_\mu u \in U^2_{\pm}(I, L^2) $ for all $\mu\ge 1$, with the norm 
$$  \norm{u}_{X^s_\pm } =
\left( \sum_{\mu\ge 1} \mu^{2s} \norm{  \mathbb{1}_I P_\mu u }_{ U^2_{\pm } }^2  \right)^{ \frac{1}{2} } < \infty,$$
 \text{where } $$\norm{f}_{ U^2_{\pm } } = \norm{S_m(\mp t) f}_{ U^2} .$$

The Duhamel representation of \eqref{Dirac3}-\eqref{Data3} is given by
 \begin{equation}\label{IntegralEq}
\psi^\pm(t) = S_m (\pm t) \psi_0^\pm  + J_{m,\pm} (\psi)(t),
\end{equation}
where 
\begin{equation}\label{Duhamel}
J_{m,\pm}(\psi)(t) =  \Pi^\pm_m (D)\int_0^t S_m(\pm(t-t'))  \left[(V \ast \innerprod{\beta\psi}{\psi} )\beta \psi\right] (t') dt'.
\end{equation}

 The linear part of \eqref{IntegralEq} satisfies the following estimate:
 \begin{equation}\label{smpsi}
\begin{split}
\norm{S_m(\pm t) \psi_0^\pm }^2_{X^s_\pm}&=\sum_{\mu\ge 1} \mu^{2s} \norm{ \mathbb{1}_I  S_m (\pm t) P_\mu \psi_0^\pm }_{ U^2_{\pm } }^2 
\\
&=\sum_{\mu\ge 1} \mu^{2s} \norm{ \mathbb{1}_I  P_\mu \psi_0^\pm }_{ U^2}^2
\\
& \sim \norm{\psi_0^\pm}^2_{H^s}.
\end{split}
 \end{equation}

 So it remains to estimate $J_{m,\pm}(\psi)(t)$. 
To this end we let $\epsilon=(\epsilon_1, \epsilon_2, \epsilon_3)$, where $\epsilon_j \in \{+, -\}$. Since $\psi= \psi^+ + \psi^-,$ where $ \psi^\pm=\Pi^\pm_m (D) \psi $, we can write
$$
J_{m,\pm}(\psi)(t) = \sum _{\epsilon_j \in \{+, -\}} J^\epsilon_\pm(\psi)(t),
$$
 where
\begin{equation}\label{Jpm}
J_{m,\pm}^\epsilon(\psi)(t) = i
\Pi^\pm_m (D)\int_0^t S_m(\pm(t-t'))  \left[(V \ast \innerprod{\beta \psi^{\epsilon_1}}{\psi^{\epsilon_2}} )\beta \psi^{\epsilon_3}\right] (t') dt'.
\end{equation}

 Theorem \ref{mainthm1} will follow by a contraction argument from \eqref{smpsi} and 
 the following cubic estimates for $J_{m,\pm}(\psi)(t)$ (see Subsection 5.3 below).
\begin{proposition}\label{Prop-WpEst}
Let $m\ge 0$ and $s>0$. For all $ \psi^\pm \in X^s_\pm $, we have
 \begin{align*}
 \norm{J_{m,\pm}^\epsilon( \psi)}_{X^s_\pm} \lesssim \prod_{j=1}^3  \norm{\psi^{\epsilon_j} }_{X^s_{\epsilon_j}}.
 \end{align*}
 \end{proposition}

\subsection{Reduction of Proposition \ref{Prop-WpEst} to dyadic quadrilinear estimates}
Due to time reversibility, we may assume that $\psi^\pm(t)=0$ for $t<0$. Let
$$
\Psi^\epsilon=V \ast \innerprod{\beta \psi^{\epsilon_1}}{\psi^{\epsilon_2}} \beta \psi^{\epsilon_3}.
$$
By definition of $U^2_\pm$, \eqref{Jpm}, Theorem \ref{thm-dual} and Proposition \ref{prop-dual}, we have
\begin{equation}
\label{Jm-X}
\begin{split}
  \norm{ P_{\lambda} J_{m,\pm}^\epsilon (\psi) }_{ U^2_{\pm } }
 &= \norm{ S_m(\mp t) P_{\lambda} J_{m,\pm}^\epsilon (\psi) }_{ U^2 } 
 \\
 &=\norm{  \Pi^\pm_m (D) P_\lambda \int_0^t S_m(\mp t')  \Psi^\epsilon(t') \, dt' }_{ U^2 }
 \\
 &=\sup_{ \norm{ \phi_{\lambda}}_{ V^2}=1}  \Big|  B\left(\Pi^\pm_m (D) P_\lambda \int_0^t S_m(\mp t')  \Psi^\epsilon (t') \, dt',  \phi   \right)
    \Big| 
 \\
 &=\sup_{ \norm{ \phi_{\lambda}}_{ V^2}=1}  \Big|\int_{\R^4}  \innerprod{ \Psi^\epsilon}{    S_m(\pm t)  P_\lambda  \phi^\pm} 
    \ dt dx\Big| 
 \\
 &=\sup_{ \norm{ \phi^\pm_{\lambda}}_{ V^2_{\pm }}=1}  \Big|\int_{\R^4}   V\ast \innerprod{\beta \psi^{\epsilon_1}}{\psi^{\epsilon_2}}   \innerprod{\beta \psi^{\epsilon_3}}{   \phi^\pm_\lambda} 
    \ dt dx\Big| ,
   \end{split}
\end{equation}
where Theorem \ref{thm-dual} and Proposition \ref{prop-dual} are used to obtain the third and fourth equality.
Hence
\begin{equation}
\label{Jm-X}
\begin{split}
 \norm{J_{m,\pm}^\epsilon(\psi)}^2_{X^s_\pm} &
 =\sum_{\lambda_4 \ge 1} \lambda_4^{2s} \norm{ P_{\lambda_4} J_{m,\pm}^\epsilon (\psi) }_{ U^2_{\pm } }^2 
\\ 
 &= \sum_{ \lambda_4 \ge 1 } \lambda_4^{2s} 
 \sup_{ \norm{ \phi^\pm_{\lambda_4}}_{ V^2_{\pm }}=1}  \Big|\int_{\R^4}  V\ast \innerprod{\beta \psi^{\epsilon_1}}{\psi^{\epsilon_2}}   \innerprod{\beta \psi^{\epsilon_3}}{   \phi^\pm_{\lambda_4}}
    \ dt dx\Big| ^2
 \\
 &\lesssim \sum_{ \lambda_4 \ge 1 } \lambda_4^{2s} 
 \sup_{ \norm{ \phi^\pm_{\lambda_4}}_{ V^2_{\pm }}=1} \left( \sum_{\lambda_1, \lambda_2, \lambda_3\ge 1}
 \Big|\int_{\R^4}  V\ast \innerprod{\beta \psi^{\epsilon_1}_{\lambda_1}}{\psi^{\epsilon_2}_{\lambda_2}}   \innerprod{\beta \psi^{\epsilon_3}_{\lambda_3}}{   \phi^\pm_{\lambda_4}} \ dt dx\Big|\right)^2
   \end{split}
\end{equation}
Set $\epsilon_4:=\pm$ and 
$$
I_m^\epsilon(\lambda) :=\Big|\int_{\R^4}  V\ast \innerprod{\beta \psi^{\epsilon_1}_{\lambda_1}}{\psi^{\epsilon_2}_{\lambda_2}}   \innerprod {\beta \psi^{\epsilon_3}_{\lambda_3}}{   \phi^{\epsilon_4}_{\lambda_4}} \ dt dx\Big|.
$$
Observe that if $\xi_j$ and  $\xi_4$  are the spatial Fourier variables for the functions $ \psi^{\epsilon_j}_{\lambda_j}$
and  $ \phi^{\epsilon_4}_{\lambda_4}$ one can see using Plancherel that the integral on the right vanishes unless 
$$
 \xi_1-\xi_2+ \xi_3-\xi_4=0.
$$
Consequently, for each $ j =1, \cdots , 4$ it follows from triangle inequality that the following conditions must be satisfied:
\begin{equation}
\label{lambcomp}
\lambda_j \le 3 \max \{\lambda_k: \  k\neq j , \ k=1, \cdots , 4 \}.
\end{equation}
Moreover, if $\xi_0$ is the frequency variable for $\innerprod{\beta \psi^{\epsilon_1}_{\lambda_1}}{\psi^{\epsilon_2}_{\lambda_2}} $ we have
\begin{equation*}
 \xi_0 =\xi_1-\xi_2=-\xi_3+\xi_4.
\end{equation*}
Thus, if $\xi_0$ has dyadic size $\mu$  it follows from triangle inequality that the following conditions must be satisfied:
\begin{equation}\label{mucomp}
\left\{
\begin{aligned}
\mu & \ll \lambda_1\sim \lambda_2\quad  \ \text{or}
\quad 
 \mu   \sim \lambda_1 \vee \lambda_2, 
 \\
\mu & \ll \lambda_3 \sim \lambda_4\quad  \ \text{or}
\quad 
 \mu   \sim \lambda_4 \vee \lambda_4.
\end{aligned}
\right.
\end{equation}

We denote the minimum,  median and maximum of $(\lambda_1, \lambda_2, \lambda_3)$ by $\lambda_{\mathrm{min}}$,
$\lambda_{\mathrm{med}}$ and $\lambda_{\mathrm{max}}$, respectively.
 \begin{lemma}\label{KeyLemma1}
Assume  $\lambda_j\ge 1$ and $\delta>0$. Then for all
$ \psi^{\pm}_{\lambda_j}\in U^2_{\pm} $ and $\phi^{\pm}_{\lambda_4} \in V^2_{\pm} $ we have 
\begin{align*}
 I_m^\epsilon(\lambda)  & \lesssim 
   \lambda_{\mathrm{med}}
^{\delta}  \prod_{j=1}^3  \norm{\psi^{\epsilon_j}_{\lambda_j} }_{U^2_{\epsilon_j}}  \norm{\phi^{\epsilon_4}_{\lambda_4} }_{V^2_{\epsilon_4}}.
   \end{align*}

 \end{lemma}
The proof of Lemma \ref{KeyLemma1} is given in Section \ref{secKeyLemma1}.

Now if we set
$
c_{j, \lambda_j}:= \norm{\psi^{\epsilon_j}_{\lambda_j} }_{U^2_{\epsilon_j}} 
$
  by definition $$ \norm{ \lambda_j^s c_{j, \lambda_j}  }_{ l^{2 }_{\lambda_j} }=\norm{\psi^{\epsilon_j} }_{X^s_{\epsilon_j}}.$$
Consequently, Proposition \ref{Prop-WpEst} follows from \eqref{Jm-X},  Lemma \ref{KeyLemma1} and the following Lemma.
\begin{lemma}\label{lemma-summing} 
Let $s> \delta>0$. Then for all $ c_{j, \lambda_j} \in l^{2 }_{\lambda_j}$ we have
  \begin{align*}
S:&=\sum_{ \lambda_4 \ge 1 }
  \left[ \sum_{  \lambda_1, \lambda_2, \lambda_3\ge 1
  }   \lambda_4^s  \lambda_{\mathrm{med}}
^{\delta}  \cdot
 c_{1, \lambda_1} c_{2, \lambda_2} c_{3, \lambda_3} \right]^2
 \\
 &\qquad \quad \lesssim \prod_{j=1}^3   \norm{  \lambda_j^{s} c_{j, \lambda_j}  }^2_{ l^{2 }_{\lambda_j} }.
 \end{align*}

\end{lemma}

\subsection{Proof of Lemma \ref{lemma-summing} }
Without loss of generality one could assume $\lambda_1 \le \lambda_2 \le \lambda_3$.
We deal with the cases $\lambda_4 \sim \lambda_3$, $\lambda_4\gg \lambda_3$ and  $\lambda_4 \ll \lambda_3$, separately.

\subsubsection{Case 1: $\lambda_4\sim \lambda_3$
} In this case we have
\begin{align*}
S&\lesssim \sum_{ \lambda_4 \ge 1 }
  \left[ \sum_{  \lambda_1, \lambda_2, \lambda_3\ge 1
  }   \lambda_4^s (\lambda_1\lambda_2)
^{\delta}  \cdot
 c_{1, \lambda_1} c_{2, \lambda_2} c_{3, \lambda_3}\right]^2
 \\
 &\lesssim 
 \norm{ \lambda_1^s c_{1, \lambda_1}  }^2_{ l^{2 }_{\lambda_1} } \norm{ \lambda_2^s c_{2, \lambda_2}  }^2_{ l^{2 }_{\lambda_2} } 
 \sum_{ \lambda_4 \ge 1 }
  \left[ \sum_{  \lambda_3\sim \lambda_4
  }  \lambda_4^s c_{3, \lambda_3} \right]^2
  \\
 &\lesssim \prod_{j=1}^3   \norm{ \lambda_j^s c_{j, \lambda_j}  }^2_{ l^{2 }_{\lambda_j} },
 \end{align*}
where to obtain the second inequality, we used Cauchy-Schwarz in $\lambda_1$ 
and in $\lambda_2$, and the fact that $ \sum_{ \lambda_j\ge 1}  \lambda_j^{-2(s-\delta)} \lesssim 1$,
 since $s>\delta$.
\subsubsection{Case 2: $\lambda_4 \gg \lambda_3$} We further divide this case into
 $\lambda_1\ll \lambda_2$ and $\lambda_1\sim \lambda_2$ 
Assume first $\lambda_1\ll \lambda_2$. Then in view of \eqref{lambcomp} we have $\lambda_4 \sim \lambda_2$. Now we can use Cauchy-Schwarz in $\lambda_1$ 
and in $\lambda_3$ to obtain
\begin{align*}
S&\lesssim \sum_{ \lambda_4 \ge 1 }
  \left[ \sum_{  \lambda_1, \lambda_2, \lambda_3\ge 1
  }   \lambda_4^s (\lambda_1\lambda_3)
^{\delta}  \cdot
 c_{1, \lambda_1} c_{2, \lambda_2} c_{3, \lambda_3}\right]^2
 \\
 &\lesssim 
 \norm{ \lambda_1^s c_{1, \lambda_1}  }^2_{ l^{2 }_{\lambda_1} } \norm{ \lambda_2^s c_{3, \lambda_3}  }^2_{ l^{2 }_{\lambda_3} } 
 \sum_{ \lambda_4 \ge 1 }
  \left[ \sum_{  \lambda_2\sim \lambda_4
  }  \lambda_4^s c_{2, \lambda_2} \right]^2
  \\
 &\lesssim \prod_{j=1}^3   \norm{ \lambda_j^s c_{j, \lambda_j}  }^2_{ l^{2 }_{\lambda_j} },
 \end{align*}

Next assume $\lambda_1\sim \lambda_2$.
 In view of \eqref{lambcomp} we have $\lambda_4 \lesssim  \lambda_2$. 
Then we apply Cauchy-Schwarz in $\lambda_1\sim \lambda_2$ 
and in $\lambda_3$ to obtain
\begin{align*}
S&\lesssim \sum_{ \lambda_4 \ge 1 }
  \left[ \sum_{  \lambda_1, \lambda_2, \lambda_3\ge 1
  }   \lambda_4^s (\lambda_1\lambda_3)
^{\delta}  \cdot
 c_{1, \lambda_1} c_{2, \lambda_2} c_{3, \lambda_3}\right]^2
 \\
 &\lesssim \left[
 \sum_{ \lambda_4 \ge 1 }
  \lambda_4^{2s} \lambda_4^{2(\delta-2s}) \right]\prod_{j=1}^3   \norm{ \lambda_j^s c_{j, \lambda_j}  }^2_{ l^{2 }_{\lambda_j} }
  \\
 &\lesssim \prod_{j=1}^3   \norm{ \lambda_j^s c_{j, \lambda_j}  }^2_{ l^{2 }_{\lambda_j} },
 \end{align*}

 \subsubsection{Case 3:  $\lambda_4\ll \lambda_3$} As above 
 we further divide this case into
 $\lambda_1\ll \lambda_2$ and $\lambda_1\sim \lambda_2$.
 Assume first $\lambda_1\ll \lambda_2$.  In view of \eqref{lambcomp} we have $\lambda_2\sim \lambda_3$. Then applying Cauchy-Schwarz first in $\lambda_1$ 
and then in $\lambda_2\sim \lambda_3$ we obtain
\begin{align*}
S&\lesssim \sum_{ \lambda_4 \ge 1 }
  \left[ \sum_{  \lambda_1, \lambda_2, \lambda_3\ge 1
  }   \lambda_4^s (\lambda_1\lambda_2)
^{\delta}  \cdot
 c_{1, \lambda_1} c_{2, \lambda_2} c_{3, \lambda_3}\right]^2
 \\
 &\lesssim \left[
 \sum_{ \lambda_4 \ge 1 }
  \lambda_4^{2s} \lambda_4^{2(\delta-2s}) \right]\prod_{j=1}^3   \norm{ \lambda_j^s c_{j, \lambda_j}  }^2_{ l^{2 }_{\lambda_j} }
  \\
 &\lesssim \prod_{j=1}^3   \norm{ \lambda_j^s c_{j, \lambda_j}  }^2_{ l^{2 }_{\lambda_j} }.
 \end{align*}

Assume next $\lambda_1\sim \lambda_2$.  By \eqref{lambcomp} we have $\lambda_3\lesssim \lambda_2$. Applying Cauchy-Schwarz first in $\lambda_1\sim \lambda_2$ and then in $\lambda_3$ we obtain
\begin{align*}
S&\lesssim \sum_{ \lambda_4 \ge 1 }
  \left[ \sum_{  \lambda_1, \lambda_2, \lambda_3\ge 1
  }   \lambda_4^s (\lambda_1\lambda_3)
^{\delta}  \cdot
 c_{1, \lambda_1} c_{2, \lambda_2} c_{3, \lambda_3}\right]^2
 \\
 &\lesssim \left[
 \sum_{ \lambda_4 \ge 1 }
  \lambda_4^{2s} \lambda_4^{2(\delta-2s}) \right]\prod_{j=1}^3   \norm{ \lambda_j^s c_{j, \lambda_j}  }^2_{ l^{2 }_{\lambda_j} }
  \\
 &\lesssim \prod_{j=1}^3   \norm{ \lambda_j^s c_{j, \lambda_j}  }^2_{ l^{2 }_{\lambda_j} }.
 \end{align*}

\subsection{Proof of Theorem \ref{mainthm1}
}
We solve the integral equation \eqref{IntegralEq} by contraction mapping techniques as follows.  
Define the mapping
\begin{equation}\label{contrmap}
\psi^\pm(t) = \mathcal T(\psi^\pm ) (t):=S_m (\pm t) \psi_0^\pm  + i J_{m,\pm} (\psi)(t).
\end{equation}
We look for the solution in the set \[ D_\delta = \left\{ \psi^\pm \in X^s_\pm : \ \
\| \psi^\pm\|_{X^s_\pm} \leq \delta \right\}. \]
For $\psi^\pm \in D_\delta$ and initial data of size  $\norm{\psi_0^\pm }_{H^s}\le \varepsilon\ll \delta$, 
we have by Proposition \ref{Prop-WpEst}
$$
 \| \mathcal T(\psi^\pm)\|_{ X^s_\pm} \lesssim \varepsilon + \delta^3 \leq \delta 
$$
for small enough $\delta$. Moreover, for solutions $\psi^\pm$ and $\phi^\pm$ with the same data, 
one can show the difference estimate
\[\begin{aligned} \|\mathcal T(\psi^\pm)-\mathcal T(\phi^\pm)\|_{X^s_\pm} &\lesssim \left(\| \psi^\pm\|_{X^s_\pm}+\|\phi^\pm\|_{ X^s_\pm} \right)^2\|\psi^\pm-\phi^\pm\|_{X^s_\pm} \\ 
&\lesssim \delta^2 \|\psi^\pm-\phi^\pm\|_{X^s_\pm} \end{aligned}\]
whenever $\psi^\pm,  \ \phi^\pm \in D_\delta$.
Hence $\mathcal T $ is a contraction on $D_\delta$ when $\delta \ll 1$, 
which implies the existence of a unique fixed point 
in $D_\delta$ solving the integral equation \eqref{contrmap}.

It thus remains to show scattering of solution of \eqref{contrmap} to a free solution
as $t\rightarrow \infty$. By Proposition \ref{Prop-Vp} and Proposition \ref{Prop-WpEst}, 
we have for each $\mu$
$$
 S_m(\mp t) P_\mu  J_{m,\pm}(\psi) \in V^2_{-, \text{rc}}
$$
and hence the limit as $t\rightarrow \infty$ exists for each $\mu$. Combining this with 
$$
\sum_{\mu\ge 1} \mu^{2s} \norm{ P_\mu  J_{m,\pm}  (\psi)}_{ V^2}^2 \lesssim 1
$$
gives
$$
\lim_{t\rightarrow \infty} S_m(\mp t) P_\mu J_{m,\pm}(\psi):= f_\pm  \in H^s.
$$
Hence for the solution $\psi^\pm $ we have 
$$
\norm {S_m(\pm t) f_\pm -\psi^\pm(t)}_{ H^s} \rightarrow 0  \ \ \text{as} \ t\rightarrow \infty.
$$

%%%%%%%%%%%%%%%%%%%%%%%%%%%%%%%%%%%%%%%%%

%%%%%%%%%%%%%%%%%%%%%%%%%%%%%%%%%%%%%%%%%
\section{Proof of Lemma \ref{KeyLemma1}}\label{secKeyLemma1}
We use the notation
$$
\psi_1:=\psi^{\epsilon_1}_{ \lambda_1}, \quad  \psi_2:=\psi^{\epsilon_2}_{ \lambda_2}, \quad \psi_3:=\psi^{\epsilon_3}_{ \lambda_3},  \quad \psi_4:=\phi^{\epsilon_4}_{ \lambda_4}.
$$
By symmetry we may set
$\epsilon_1=\epsilon_3=+$ in the integral for $I^\epsilon_m$, and thus 
we need to estimate
 $$
I(\lambda):=I^\epsilon_m(\lambda)=  \Big|\int_{\R^{4}}   V\ast \innerprod{\beta \psi_1}{\psi_2}   \cdot \innerprod{\beta \psi_3}{ \psi_4}
  \ dt dx\Big|
$$
with $\epsilon_1=\epsilon_3=+$.

Let $Q$ denote any one of 
$Q_l$'s for $l=1, \cdots, 3$ and $B$ denote any one of $B_l$'s for $l=1, \cdots, 4$.
In view of the equations in \eqref{pp}--\eqref{qb}, it suffices to show for $\epsilon_1=\epsilon_3=+$ the estimates
\begin{equation}\label{keyest}
  I_k(\lambda) \lesssim   \lambda_{\mathrm{med}}
^{\delta} \prod_{j=1}^3  \norm{\psi_j }_{U^2_{\epsilon_j}}  \norm{\psi_4 }_{V^2_{\epsilon_4}} \quad ( k=1, \cdots, 4),
\end{equation}
where
\begin{align*}
I_1(\lambda)&=\Big| \int_{\R^{4}}   V\ast Q( \psi_1, \psi_2)   \cdot Q ( \psi_3,   \psi_4)
  \ dt dx \Big| ,
  \\
I_2 (\lambda)&= \Big| \int_{\R^{4}}    V\ast B(  \psi_1, \psi_2)    \cdot  Q(\psi_3,   \psi_4 )    \ dt dx \Big| ,
\\
I_3(\lambda)&= \Big| \int_{\R^{4}}    V\ast Q(  \psi_1, \psi_2)    \cdot  B(\psi_3,   \psi_4)   \ dt dx \Big|,
\\
I_4(\lambda)&= \Big| \int_{\R^{4}}   V\ast B( \psi_1, \psi_2)    \cdot B( \psi_3,   \psi_4)   \ dt dx \Big|.
\end{align*}

%%%%%%%%%%%%%%%%%%%%%%%%%%%%%%%%%%%%%%%%%

In the arguments that follow we repeatedly use the following facts (see Propositions \ref{Prop-Up}  and \ref{Prop-Vp}): 
\begin{equation}\label{ubv2emb}
 U^2_\pm \subset U^p_\pm , \quad V^2_\pm \subset U^p_\pm  \quad \text{for} \ \ p>2.
\end{equation}
We shall also use the conditions in \eqref{mucomp}.
We remark that in $\R^3$ convolution with $V(x)=e^{-|x|}/|x|$ is (up to a multiplicative constant) the Fourier-multiplier $\angles{D}^{-2}$ 
with symbol $\angles{\xi}^{-2}$.

%%%%%%%%%%%%%%%%%%%%%%%%%%%%%%%%%%%%%%%%%

\subsection{Estimate for $I_4(\lambda)$}
By the symmetry of our argument we may assume $\lambda_1 \le \lambda_2$ and  $\lambda_3 \le \lambda_4$. 
Using Littlewood-Paley decomposition,  H\"{o}lder and  the bilinear estimate\eqref{BiestB}, we obtain
\begin{align*}
  I_4(\lambda) & \lesssim  \sum_{\mu\ge 1 }  \norm{\angles{D}^{-2}    P_\mu   B( \psi_1, \psi_2 )}  \norm{  P_\mu  B(\psi_3 , \psi_4)  } 
\\
  &
  \lesssim  \sum_{\mu\ge 1 } \angles{\mu}^{-2} (\lambda_1 \lambda_3)^{-\frac12}(\lambda_2\lambda_4)^{\frac12}
\prod_{j=1}^4 \norm{\psi_j}_{V^2_{\epsilon_j}}  
\\
&\lesssim \prod_{j=1}^3  \norm{\psi_j}_{U^2_{\epsilon_j}}  \norm{\psi_4 }_{V^2_{\epsilon_4}}
   \end{align*}
    where to sum up the third line we considered the following cases: $\lambda_1\sim \lambda_2$ or $\lambda_1\ll \lambda_2\sim \mu $ and $\lambda_3\sim \lambda_4$ or $\lambda_3\ll \lambda_4 \sim \mu $.

    %%%%%%%%%%%%%%%%%%%%%%%%%%%%%%%%%%%%%%%%%

\subsection{Estimate for $I_3(\lambda)$}
As in the previous subsection we may assume $\lambda_3 \le \lambda_4$. Then
using Littlewood-Paley decomposition,  H\"{o}lder,  the null-form estimates in Corollary \ref{lmNullestU2} and the bilinear estimate \eqref{BiestB}, we obtain
\begin{align*}
  I_3(\lambda) & \lesssim  \sum_{\mu\ge 1 }\norm{\angles{D}^{-2}    P_\mu   Q( \psi_1, \psi_2 )}   \norm{  P_\mu  B(\psi_3 , \psi_4)  }
\\
  & \lesssim \sum_{\mu\ge 1 } \angles{\mu}^{-2} \mu\lambda_3^{-\frac12}  \lambda_4^{\frac12}\prod_{j=1}^2  \norm{\psi_j}_{U^2_{\epsilon_j}} \prod_{j=3}^4 \norm{\psi_j}_{V^2_{\epsilon_j}}  
  \\
  & \lesssim \prod_{j=1}^3  \norm{\psi_j}_{U^2_{\epsilon_j}}  \norm{\psi_4 }_{V^2_{\epsilon_4}}
   \end{align*}
where to sum up  the second line we considered the cases $\lambda_3\sim \lambda_4$ or $\lambda_3\ll \lambda_4 \sim \mu $.

%%%%%%%%%%%%%%%%%%%%%%%%%%%%%%%%%%%%%%%%%%%%%

\subsection{Estimate for $I_2(\lambda)$}
 By the symmetry of our argument we may assume $\lambda_1 \le\lambda_2$ and $\lambda_3 \le\lambda_4$.

 \subsubsection{ \textbf{Case} $\lambda_3 \ll \lambda_4\sim \mu $}  As in the preceding subsections we
 use H\"{o}lder and the bilinear estimates \eqref{BiestQV2} and \eqref{BiestB} to obtain
\begin{align*}
I_{2}(\lambda)
&\lesssim \sum_{\mu\ge 1 }\norm{\angles{D}^{-2}    P_\mu   B( \psi_1, \psi_2 )}\norm{  P_\mu  Q ( \psi_3,\psi_4 ) }
\\
  &
  \lesssim  \sum_{\mu\ge 1 } \angles{\mu}^{-2}  \lambda_1^{-\frac12}(\lambda_2\lambda_3\lambda_4)^{\frac12}
\prod_{j=1}^4 \norm{\psi_j}_{V^2_{\epsilon_j}} 
\\
&\lesssim \prod_{j=1}^3  \norm{\psi_j}_{U^2_{\epsilon_j}}  \norm{\psi_4 }_{V^2_{\epsilon_4}},
   \end{align*}
where to sum up the second line we considered the cases $\lambda_1\sim \lambda_2$ or $\lambda_1\ll \lambda_2\sim \mu $.
 %**************************************************************************  

\subsubsection{\textbf{Case} $\lambda_3 \sim \lambda_4$}  
\subsubsection*{\textbf{Sub-case 1:} $\lambda_2 \gtrsim \lambda_3\sim \lambda_4$}   
 Then by H\"{o}lder,  the bilinear estimate \eqref{BiestB} and the null form estimates in Corollary \ref{lmNullestU2} we obtain
 \begin{align*}
I_{2}(\lambda)
&\lesssim \sum_{\mu\ge 1 }\norm{\angles{D}^{-2}    P_\mu   B( \psi_1, \psi_2 )}\norm{  P_\mu  Q ( \psi_3,\psi_4 ) }
\\
  &\lesssim  \sum_{\mu\ge 1 } \angles{\mu}^{-2}\mu \lambda_1^{-\frac12}   \lambda_2^{\frac12}  \prod_{j=1}^2 \norm{\psi_j }_{V^2_{\epsilon_j}} \prod_{j=3}^4 \norm{\psi_j }_{U^2_{\epsilon_j}}  
  \\
  &\lesssim  \prod_{j=1}^3  \norm{\psi_j }_{U^2_{\epsilon_j}}  \norm{\psi_4 }_{U^2_{\epsilon_4}},
   \end{align*}
   where we used $\lambda_1 \sim \lambda_2$ or $\lambda_1 \ll \lambda_2\sim \mu$ to sum up the second line.

On the other hand,  applying H\"{o}lder and the bilinear estimates \eqref{BiestQU4} and \eqref{BiestB} we obtain
\begin{align*}
I_{2} (\lambda)
  &\lesssim \sum_{\mu\ge 1 } \angles{\mu}^{-2} \lambda_1^{-\frac12}   (\lambda_2 \lambda_3\lambda_4)^{\frac12}   \prod_{j=1}^2 \norm{\psi_j }_{V^2_{\epsilon_j}}   \prod_{j=3}^4 \norm{\psi_j }_{U^4_{\epsilon_j}}     
  \\
  &\lesssim  \lambda_3 \prod_{j=1}^3 \norm{\psi_j }_{U^2_{\epsilon_j}}  \norm{\psi_4 }_{U^4_{\epsilon_4}},
   \end{align*}
    where we used $\lambda_1 \sim \lambda_2$ or $\lambda_1 \ll \lambda_2\sim \mu$ to sum up the first line.

  Now we use Lemma \ref{lm-interpu2ub} to interpolate between 
   the two estimates for $I_2(\lambda)$ above and obtain
\begin{align*}
  I_2(\lambda) \lesssim  \lambda_3^{\delta}  \prod_{j=1}^3 \norm{\psi_j}_{U^2_{\epsilon_j}}  \norm{\psi_4 }_{V^2_{\epsilon_4}}.
   \end{align*}

 %**************************************************************************  

   \subsubsection*{\textbf{Sub-case 2:} $\lambda_2 \ll\lambda_3\sim \lambda_4$}  
Since by assumption $\lambda_1\le \lambda_2$ we have $\mu\ll \lambda_3\sim \lambda_4.$  We separate this sub-case further into (i): $\epsilon_4=+$ and  (ii): $\epsilon_4=-$. Recall that $\epsilon_3=+$.
   \subsubsection*{ (i):  $\epsilon_4=+$}  
 By H\"{o}lder , the bilinear estimate \eqref{BiestB} and the null form estimate in Corollary \ref{lmNullestU2}\eqref{Nullest++} we obtain
\begin{align*}
I_{2}(\lambda)
&\lesssim \sum_{\mu\ge 1 }\norm{\angles{D}^{-2}    P_\mu   B( \psi_1, \psi_2 )}\norm{  P_\mu  Q ( \psi_3,\psi_4 ) }
\\
  &\lesssim  \sum_{\mu\ge 1 } \angles{\mu}^{-2}   \lambda_1^{-\frac12}   \lambda_2^{\frac12} \cdot  \mu^\frac32 \lambda_3^{-\frac12} \prod_{j=1}^3  \norm{\psi_j}_{V^2_{\epsilon_j}}   \prod_{j=3}^4 \norm{\psi_j}_{U^2_{\epsilon_j}}
  \\
  &\lesssim  \lambda_3^{-\frac12 }  \prod_{j=1}^3  \norm{\psi_j }_{U^2_{\epsilon_j}}  \norm{\psi_4 }_{U^2_{\epsilon_4}},
   \end{align*}
 where we used $\lambda_1 \sim \lambda_2$ or $\lambda_1 \ll \lambda_2\sim \mu$ to sum up the second line.
 
On the other hand, similarly as in Sub-case 1 above we have
\begin{align*}
I_{2}(\lambda)
  &\lesssim  \lambda_3 \prod_{j=1}^3 \norm{\psi_j }_{U^2_{\epsilon_j}}  \norm{\psi_4 }_{U^4_{\epsilon_4}}.
   \end{align*}
 Then we use Lemma \ref{lm-interpu2ub} to interpolate between 
   the two estimates for $I_2$ above and obtain
   \begin{align*}
 I_2(\lambda) \lesssim \prod_{j=1}^3  \norm{\psi_j }_{U^2_{\epsilon_j}}  \norm{\psi_4 }_{V^2_{\epsilon_4}}.
   \end{align*}

   \subsubsection*{(ii):  $\epsilon_4=-$}  This case is contained in Lemma \ref{lm-BI-reson} below.

%'**********************************************************************************
\subsection{Estimate for $I_1(\lambda)$}
  By the symmetry of our argument we may assume $\lambda_1 \le\lambda_2$ and  $\lambda_3 \le\lambda_4$.

 %********************************************************************************

 \subsubsection{\textbf{Case}  $\lambda_3 \ll \lambda_4 \sim \mu$}
By H\"{o}lder, \eqref{BiestB} and the null form estimates in Corollary \ref{lmNullestU2} we obtain
\begin{align*}
I_{1}(\lambda)& \lesssim \sum_{\mu\ge 1 }\norm{\angles{D}^{-2}    P_\mu   Q( \psi_1, \psi_2 )}  \norm{  P_\mu  Q(\psi_3 , \psi_4)  }
\\
&\lesssim   \sum_{\mu\sim \lambda  } \angles{\mu}^{-2} \mu \lambda_3^{\frac12} \lambda_4^{\frac12} 
\prod_{j=1}^4  \norm{\psi_j}_{U^2_{\epsilon_j}}  \prod_{j=4}^3 \norm{\psi_j}_{V^2_{\epsilon_j}}  
\\
&\lesssim  
\prod_{j=1}^3  \norm{\psi_j }_{U^2_{\epsilon_j}}  \norm{\psi_4 }_{V^2_{\epsilon_4}} .  
   \end{align*}

%********************************************************************************

\subsubsection{\textbf{Case} $\lambda_3 \sim \lambda_4 \gtrsim \mu $}  
\subsubsection*{\textbf{ Sub-case 1:}  $\lambda_2 \gtrsim \lambda_3\sim \lambda_4$}  
   By H\"{o}lder  and the null form estimates in Corollary \ref{lmNullestU2} we obtain
\begin{align*}
            I_{1}(\lambda)  
            &\lesssim  \sum_{\mu\ge 1 }  \angles{\mu}^{-2}\norm{  P_\mu  Q (  \psi_1, \psi_2)   } \norm{  P_\mu  Q(\psi_3,   \psi_4)    }
           \\
            &\lesssim  \sum_{1\le \mu\lesssim \lambda_3 }  \angles{\mu}^{-2} 
            \mu^2 \prod_{j=1}^4  \norm{\psi_j }_{U^2_{\epsilon_j}} 
            \\
            &\lesssim \ln (\lambda_3) \prod_{j=1}^3  \norm{\psi_j }_{U^2_{\epsilon_j}}  \norm{\psi_4 }_{U^2_{\epsilon_4}}.
   \end{align*}

On the other hand,  by H\"{o}lder, \eqref{BiestQU4} and the null form estimates in Corollary \ref{lmNullestU2} we have
\begin{align*}
  I_{1}(\lambda)  &\lesssim  \sum_{\mu\ge 1 }  \norm{ \angles{D}^{-2} P_\mu  Q (  \psi_1, \psi_2)   }  \norm{  P_\mu  Q(\psi_3,   \psi_4)    }        
\\
  &\lesssim  \sum_{\mu\ge 1 }  \angles{\mu}^{-2}\mu ( \lambda_3\lambda_4)^\frac12
\prod_{j=1}^2  \norm{\psi_j}_{U^2_{\epsilon_j}}  \prod_{j=3}^4    \norm{\psi_j}_{U^4_{\epsilon_j}} 
\\
  &\lesssim \lambda_3
\prod_{j=1}^3  \norm{\psi_j}_{U^2_{\epsilon_j}}  \norm{\psi_4 }_{U^4_{\epsilon_4}}.
   \end{align*}
  
  Then we interpolate the two estimates for $I_1(\lambda)$ above, using Lemma \ref{lm-interpu2ub}, and obtain
      $$
      I_{1} (\lambda)
            \lesssim  \lambda_3^{\delta}
\prod_{j=1}^3  \norm{\psi_j }_{U^2_{\epsilon_j}}  \norm{\psi_4 }_{V^2_{\epsilon_4}} .
       $$

%********************************************************************************

   \subsubsection*{\textbf{Sub-case 2:} $\lambda_2 \ll\lambda_3\sim \lambda_4$}  
Since by assumption $\lambda_1\le \lambda_2$ we have $\mu\ll \lambda_3\sim \lambda_4.$  We separate this sub-case further into (i): $\epsilon_4=+$ and  (ii): $\epsilon_4=-$.  Recall that $\epsilon_3=+$.
   
   \subsubsection*{ (i):  $\epsilon_4=+$}  
By H\"{o}lder and the null form estimates in Corollary \ref{lmNullestU2}
   \begin{align*}
    I_{1}(\lambda)  &\lesssim  \sum_{\mu\ge 1 }  \norm{ \angles{D}^{-2} P_\mu  Q (  \psi_1, \psi_2)   }  \norm{  P_\mu  Q(\psi_3,   \psi_4)    }
            \\
            &\lesssim \sum_{1\le \mu\lesssim  \lambda_2 }  \angles{\mu}^{-2}   \mu^\frac52 \lambda_3^{-\frac12} \prod_{j=1}^4  \norm{\psi_j }_{U^2_{\epsilon_j}} 
              \\
            &\lesssim  \lambda_2^\frac12 \lambda_3^{-\frac12} \prod_{j=1}^3  \norm{\psi_j }_{U^2_{\epsilon_j}}  \norm{\psi_4}_{U^2_{\epsilon_4}}.
   \end{align*}

   On the other hand, by H\"{o}lder and \eqref{BiestQU4} we have
\begin{align*}
I_{1}(\lambda)& \lesssim \sum_{\mu\ge 1 }\norm{\angles{D}^{-2}    P_\mu   Q( \psi_1, \psi_2 )}   \norm{  P_\mu  Q(\psi_3 , \psi_4)  }
\\
&\lesssim  \sum_{\mu\ge 1 } \angles{\mu}^{-2} (\lambda_1\lambda_2 \lambda_3\lambda_4)^\frac12  \prod_{j=1}^4  \norm{\psi_j }_{U^4_{\epsilon_j}} 
\\
&\lesssim    \lambda_2 \lambda_3
\prod_{j=1}^3  \norm{\psi_j }_{U^2_{\epsilon_j}}  \norm{\psi_4}_{U^4_{\epsilon_4}}
   \end{align*}

We then use Lemma \ref{lm-interpu2ub} to interpolate between 
   the two estimates for $I_1(\lambda)$ above and obtain
\begin{align*}
I_{1}(\lambda)
&\lesssim  \lambda_2^{\delta}
\prod_{j=1}^3  \norm{\psi_j}_{U^2_{\epsilon_j}}  \norm{\psi_4}_{V^2_{\epsilon_4}}.
   \end{align*}

   \subsubsection*{ (ii):  $\epsilon_4=-$}  
This case is contained in Lemma \ref{lm-BI-reson} below.

%********************************************************************************

%********************************************************************************

\subsection{A modulation Lemma}
Recall
$$
\psi_1=\psi^{\epsilon_1}_{ \lambda_1}, \quad  \psi_2=\psi^{\epsilon_2}_{ \lambda_2}, \quad \psi_3=\psi^{\epsilon_3}_{ \lambda_3},  \quad \psi_4=\phi^{\epsilon_4}_{ \lambda_4},
$$
where $\epsilon_j\in \{+, -\}$ and $\epsilon_1=\epsilon_3=+$.

In the case $\epsilon_4=-$ and $\lambda_1\le \lambda_2\ll \lambda_3\sim \lambda_4$ we exploit the non-resonance structure in the integral for $I_k(\lambda)$ to establish the required estimates for $I_1(\lambda)$ and $I_2(\lambda)$ (see Subsections 6.3.2(ii) and 6.4.2(ii) above). This is contained in the following Lemma.
\begin{lemma}\label{lm-BI-reson}
Let 
$$
J(\lambda)=\Bigabs{
 \int_{\R^{4}}   V\ast A( \psi_1, \psi_2)   \cdot  Q ( \psi_3,   \psi_4 )
  \ dt dx},
$$
where $A$ is either $Q$ or $B$. Assume $\epsilon_1=\epsilon_3=+,\  \epsilon_4=-$
 and $\lambda_1\le \lambda_2\ll \lambda_3\sim \lambda_4$. Then
 \begin{equation}\label{J-resonest}
J(\lambda) \lesssim \lambda_2^{\delta}  
\prod_{j=1}^3  \norm{\psi_j}_{U^2_{\epsilon_j}}  \norm{\psi_4}_{V^2_{\epsilon_4}}.
 \end{equation}
\end{lemma}

\begin{proof}
 Decompose the functions $\psi_j$ into a low and high modulation part, i.e.,
 $\psi_j=\psi_j^l+\psi_j^h$, where
$$
\psi_j^h:= \Lambda^{\epsilon_j}_{\ge \lambda_3/8} \psi_j, \quad \psi_j^l:= \ \Lambda^{\epsilon_j}_{< \lambda_3/8 } \psi_j.
$$

We claim that
$$
\int_{\R^{4}}    V\ast A( \psi^l_1, \psi^l_2)   \cdot  Q( \psi^l_3,   \psi^l_4 )
  \ dt dx=0.
$$
Indeed, let
$(\tau_j, \xi_j)$ be the space-time Fourier variables
of the functions $\psi^l_j$.
 Clearly, the integral vanishes unless $\tau_1-\tau_2+\tau_3-\tau_4=0$ and $\xi_1-\xi_2+\xi_3-\xi_4=0$. 
 By assumption and definition of low modulation, the contributing set must then satisfy 
 \begin{align*}
 4\cdot \frac {\lambda_3} 8=\frac \lambda 2 & > | (\tau_1 + \angles{\xi_1}_m) - (\tau_2 +\epsilon_2 \angles{\xi_2}_m) +  (\tau_3 + \angles{\xi_3}_m )- (\tau_4 - \angles{\xi_4}_m )|
 \\
 &= \angles{\xi_1}_m  - \epsilon_2  \angles{\xi_2}_m +  \angles{\xi_3}_m  + \angles{\xi_4}_m   \ge  \frac{\lambda_3} 2,
 \end{align*}
 which is a contradiction,  and hence the integral vanishes.
Thus, we always have at least one function on high modulation in the integral for $J$.
There are 15 cases of which at least one of the four functions has high modulation but we consider only 4 cases where one of the functions is on high  modulation and the other functions are on high or low modulation.

\subsection*{Case 1:  $\psi_1=\psi^h_1$ or $\psi_2=\psi^h_2$}
We only consider the case $\psi_1=\psi^h_1$ since 
  the case $\psi_2=\psi^h_2$ can be handled in a similar way.
  
  \subsubsection*{Sub-case (i):  $A=Q$}
  By H\"{o}lder, Sobolev, Lemma \ref{lm-str}, Corollary \ref{lmNullestU2}\eqref{Nullest+-} and \eqref{modul2} we obtain
 \begin{align*}
 J(\lambda)  & \lesssim \sum_{\mu\ge 1 } \|  \angles{D}^{-2} P_\mu  Q (  \psi_1^h, \psi_2) \|_{  L^2_t L^1_x}  \|  P_\mu  Q (  \psi_3, \psi_4) \|_{  L^2_t L^\infty_x}   
 \\
  & \lesssim \sum_{\mu\ge 1 } \angles{\mu}^{-2}  \mu^\frac32 \norm{ \psi^h_1  }_{ L^2_{t ,x} }   \norm{ \psi_2  }_{    L^\infty_t L^2_x}    \|  P_\mu  Q (  \psi_3, \psi_4) \|_{  L^2_{t ,x} }   
  \\
 &\lesssim \sum_{1\le \mu\lesssim \lambda_2 } \angles{\mu}^{-2}  \mu^\frac52 \lambda_3^{-\frac12}
\prod_{j=1}^2  \norm{\psi_j }_{V^2_{\epsilon_j}} \prod_{j=3}^4  \norm{\psi_j }_{U^2_{\epsilon_j}} 
\\
 &\lesssim  \lambda_2^\frac12 \lambda_3^{-\frac12}
\prod_{j=1}^3  \norm{\psi_j}_{U^2_{\epsilon_j}}  \norm{\psi_4 }_{U^2_{\epsilon_4}},
 \end{align*}
 where we used the physical space representation of the term $  Q (  \psi_1^h, \psi_2) $ found in \eqref{Nullform-phy}.
On the other hand,  applying \ref{BiestQU4} to the norm $   \|  P_\mu  Q (  \psi_3, \psi_4) \|_{  L^2_{t ,x} }   $ we obtain
 \begin{align*}
 J  (\lambda)
  &\lesssim \sum_{1\le \mu\lesssim \lambda_2 } \angles{\mu}^{-2}  \mu^\frac32 \lambda_3^{-\frac12} (\lambda_3 \lambda_4)^\frac12
\prod_{j=1}^2  \norm{\psi_j }_{V^2_{\epsilon_j}}  \prod_{j=3}^4 \norm{\psi_j }_{U^4_{\epsilon_j}}
\\
 &\lesssim  \lambda_3^{\frac12}
\prod_{j=1}^3  \norm{\psi_j}_{U^2_{\epsilon_j}}  \norm{\psi_4 }_{U^4_{\epsilon_4}},
 \end{align*}
 Now we use Lemma \ref{lm-interpu2ub} to interpolate between 
   the two estimates for $J(\lambda)$  above to obtain the desired estimate \eqref{J-resonest}.

  \subsubsection*{Sub-case (ii):  $A=B$}
 By H\"{o}lder, Sobolev, Lemma \ref{lm-str}, Corollary \ref{lmNullestU2}\eqref{Nullest+-} and \eqref{modul2} we obtain
 \begin{align*}
 J(\lambda)  & \lesssim \sum_{\mu\ge 1 } \|  \angles{D}^{-2} P_\mu  B (  \psi_1^h, \psi_2) \|_{  L^2_t L^1_x}  \|  P_\mu  Q (  \psi_3, \psi_4) \|_{  L^2_t L^\infty_x}   
 \\
  & \lesssim \sum_{\mu\ge 1 } \angles{\mu}^{-2}  \mu^\frac32  
  \lambda_1^{-1}\norm{ \psi^h_1  }_{ L^2_{t ,x} }   \norm{ \psi_2  }_{    L^\infty_t L^2_x}    \|  P_\mu  Q (  \psi_3, \psi_4) \|_{  L^2_{t ,x} }   
  \\
 &\lesssim \sum_{1\le \mu\lesssim \lambda_2 } \angles{\mu}^{-2}  \mu^\frac52 \lambda_1^{-1} \lambda_3^{-\frac12}
\prod_{j=1}^2  \norm{\psi_j }_{V^2_{\epsilon_j}}  \prod_{j=3}^4  \norm{\psi_j }_{U^2_{\epsilon_j}} 
\\
 &\lesssim  \lambda_2^\frac12 \lambda_3^{-\frac12}
\prod_{j=1}^3  \norm{\psi_j}_{U^2_{\epsilon_j}}  \norm{\psi_4 }_{U^2_{\epsilon_4}},
 \end{align*}
where the  H\"{o}lder inequality is applied on the physical space representation of $  B (  \psi_1^h, \psi_2) $ found in \eqref{Biform-phy}.
On the other hand,  applying \ref{BiestQU4} to the norm $   \|  P_\mu  Q (  \psi_3, \psi_4) \|_{  L^2_{t ,x} }   $ we obtain
 \begin{align*}
 J(\lambda)
  &\lesssim \sum_{1\le \mu\lesssim \lambda_2 } \angles{\mu}^{-2}  \mu^\frac32 \lambda_1^{-1} \lambda_3^{-\frac12} (\lambda_3 \lambda_4)^\frac12
\prod_{j=1}^2  \norm{\psi_j }_{V^2_{\epsilon_j}} \prod_{j=3}^4  \norm{\psi_j}_{U^4_{\epsilon_j}}  
\\
 &\lesssim  \lambda_3^{\frac12}
\prod_{j=1}^3  \norm{\psi_j}_{U^2_{\epsilon_j}}  \norm{\psi_4 }_{U^4_{\epsilon_4}},
 \end{align*}
 Interpolating between 
   the two estimates for $J(\lambda)$  above, using Lemma \ref{lm-interpu2ub}, and obtain the desired estimate \eqref{J-resonest}.

\subsection*{Case 2: $\psi_3=\psi^h_3$  or $\psi_4=\psi^h_4$ }
We only consider the case $\psi_4=\psi^h_4$ since 
  the case $\psi_3=\psi^h_3$ can be handled in a similar way.
  
\subsubsection*{Sub-case (i):  $A=Q$}

By H\"{o}lder, Sobolev, Lemma \ref{lm-str},  Lemma \ref{lmNullestU2}  and \eqref{modul2} we have
 \begin{align*}
J(\lambda)  & \lesssim  \sum_{\mu\ge 1 } \norm{\angles{D}^{-2}    P_\mu  Q (  \psi_1, \psi_2)  }_{L^2_t L_x^\infty}   \norm{   P_\mu  Q (  \psi_3, \psi_4^h)  }_{L^2_t L_x^1} 
  \\
 &\lesssim   \sum_{\mu\ge 1 }  \mu^{-\frac12} \norm{   P_\mu  Q (  \psi_1, \psi_2)  } \norm{ \psi_3  }_{L^\infty_t L^2_x } 
 \norm{  \psi^h_4  }_{ L^2 }   
 \\
 &\lesssim  \sum_{1\le \mu \lesssim \lambda_3 } \mu^\frac12 \lambda_3^{-\frac12}  
\prod_{j=1}^2  \norm{\psi_j }_{U^2_{\epsilon_j}} \prod_{j=3}^4  \norm{\psi_j}_{V^2_{\epsilon_j}} 
\lesssim  
\prod_{j=1}^3  \norm{\psi_j }_{U^2_{\epsilon_j}}  \norm{\psi_4 }_{V^2_{\epsilon_4}}.
 \end{align*}

\subsubsection*{Sub-case (ii):  $A=B$}
By H\"{o}lder, Sobolev, \eqref{BiestB} and \eqref{modul2} we have
\begin{align*}
J(\lambda)  & \lesssim  \sum_{\mu\ge 1 } \norm{  \angles{D}^{-2}  P_\mu  B (  \psi_1, \psi_2)  }_{L^2_t L_x^\infty}  \norm{   P_\mu  Q (  \psi_3, \psi_4^h)  }_{L^2_t L_x^1}  
  \\
 &\lesssim  \sum_{\mu\ge 1 }  \mu^{-\frac12}\norm{   P_\mu  B (  \psi_1, \psi_2)  } \norm{   P_\mu  Q (  \psi_3, \psi_4^h)  }_{L^2_t L_x^1}  
 \\
 &\lesssim   \sum_{\mu\ge 1 }  \mu^{-\frac12}  \lambda_1^{-1}(\lambda_1\lambda_2)^\frac12 \norm{   \psi_1  }_{V^2_{\epsilon_1}}
  \norm{   \psi_2  }_{V^2_{\epsilon_2}}  \norm{ \psi_3  }_{L^\infty_t L^2_x } 
 \norm{  \psi^h_4  }_{ L^2_{t,x} }   
 \\
 &\lesssim  \sum_{  \mu\ge 1  }    \mu^{-\frac12} \lambda_1^{-\frac12} \lambda_2^{\frac12}\lambda_3^{-\frac12}  
\prod_{j=1}^4  \norm{\psi_j }_{V^2_{\epsilon_j}} 
\lesssim  
\prod_{j=1}^3  \norm{\psi_j }_{U^2_{\epsilon_j}}  \norm{\psi_4}_{V^2_{\epsilon_4}},
 \end{align*}
  since $\lambda_1 \le \lambda_2\ll  \lambda_3\sim \lambda_4 $.

  \end{proof}

 %**************************************************************************  
 %**************************************************************************  
  \section{Proof of Lemma \ref{lmI+-} }  \label{seclmI+-}
 
To  prove the estimates in Lemma \ref{lmI+-} we closely follow the argument of Foschi–Klainerman for $m = 0$ 
\cite[Lemma 4.1 and Lemma 4.4]{FK00}. 

For a smooth function $\varphi$, define the hypersurface
$S=\{x\in\R^3: \varphi(x)=0\}. $ If
$\nabla \varphi \neq 0$ for $x\in S \cap \text{supp} (\Phi)$, then
    \begin{equation}\label{Delta1}
   \int_{\R^3} \Phi(x) \delta(\varphi(x))\, dx=\int_S \frac{\Phi(x)}{|\nabla \varphi(x)|} \, dS_x.
  \end{equation}
For a nonnegative smooth function $h$ which does not vanish on $S$, \eqref{Delta1} also implies
 \begin{equation}\label{Delta2}
  \delta(\varphi(x))=h(x)\delta\left(h(x)\varphi(x)\right).
  \end{equation}

\subsection{Proof of Lemma \ref{lmI+-}\eqref{lmI+-i}}
First note that the integral $I_+(f, g)$ is supported on the set 
\begin{equation}\label{elipsoid} 
\mathcal E(\tau, \xi)=\{ \eta\in \R^3: \angles{\eta}_m+\angles{\xi-\eta}_m =\tau \}.
\end{equation}
Thus for $ \eta\in \mathcal E(\tau, \xi)$ we have
 \begin{equation}\label{elipest} 
\begin{split}
\tau^2-|\xi|^2-4m^2&= (\angles{\eta}_m+\angles{\xi-\eta}_m)^2-|\xi|^2-4m^2
\\
&=
 2\angles{\eta}_m\angles{\xi-\eta}_m-2(|\eta||\xi-\eta|+m^2) + 2(|\eta||\xi-\eta|-\eta \cdot (\xi-\eta))
 \\
 &\ge 0.
\end{split}
\end{equation}

Now we use \eqref{Delta2} to write
\begin{equation}\label{deltacomp} 
\left\{
\begin{aligned}
  &\delta( \tau-\angles{\eta}_m-\angles{\xi-\eta}_m)
  \\  
   & \qquad =[( \tau-\angles{\eta}_m)+ \angles{\xi-\eta}_m]
   \delta\left( (\tau-\angles{\eta}_m)^2-
   \angles{\xi-\eta}_m^2\right)
   \\
      & \qquad =2(\tau- \angles{\eta}_m)
   \delta\left(\tau^2-|\xi|^2-2\tau \angles{\eta}_m +2\xi \cdot \eta \right),
 \end{aligned}
\right.
\end{equation}
  where in the first line we multiplied the argument of the delta function on the left by 
   $ (\tau-\angles{\eta}_m)+\angles{\xi-\eta}_m$.
 
Introduce polar coordinate $\eta=\sigma \omega$, where $\omega \in \mathbb S^2$: $$ |\eta|=\sigma, \quad
d\eta= \sigma^2 dS_w d\sigma.$$
Then using \eqref{deltacomp} and \eqref{Delta2} we obtain
\begin{align*}
I_+(f, g)(\tau, \xi)&=  2 \int_{\R^3} (\tau- \angles{\eta}_m) f(|\eta|) g( |\xi-\eta|) 
   \delta\left(\tau^2-|\xi|^2-2\tau \angles{\eta}_m +2\xi \cdot \eta \right) \, d\eta
\\
&\simeq \int_0^{\sqrt{\tau^2-m^2}} \int_{\omega \in \mathbb S^2 } 
\sigma^2(\tau-\angles{\sigma}_m)  f(\sigma)g \left( \sqrt{ (\tau-\angles{\sigma}_m)^2-m^2 }\right)
\delta\left( \tau^2-|\xi|^2-2\tau \angles{\sigma}_m+2\sigma \xi\cdot \omega \right )\, dS_\omega d\sigma
\\
&\simeq \frac1{|\xi|} \int_0^{\sqrt{\tau^2-m^2}} \int_{\omega \in \mathbb S^2 } 
\sigma (\tau-\angles{\sigma}_m)   f(\sigma)g \left( \sqrt{ (\tau-\angles{\sigma}_m)^2-m^2 }\right)
\delta\left( \frac{\tau^2-|\xi|^2-2\tau \angles{\sigma}_m}{2|\xi|\sigma }+ \frac{\xi  \cdot \omega}{|\xi|} \right )\, dS_\omega d\sigma,
 \end{align*}
 where in the second inequality we used the fact that
  $$ \tau-\angles{\eta}_m)=\angles{\xi-\eta}_m \ \Rightarrow \  |\xi-\eta|= \sqrt{ (\tau-\angles{\sigma}_m)^2-m^2 } .$$
 
Changing variable 
$$s=  \frac{\omega \cdot\xi}{ |\xi|} \ \Rightarrow \  dS_w = dS_{w'} ds, \ \ \text{
where } \ \ \omega' \in \mathbb S^1,$$ 
we have
\begin{align*}
I_+(f, g)(\tau, \xi)
&\simeq \frac1{|\xi|}\int_0^{\sqrt{\tau^2-m^2}} \int_{-1}^{1} 
\sigma (\tau-\angles{\sigma}_m)  f(\sigma) g\left( \sqrt{ (\tau-\angles{\sigma}_m)^2-m^2 }\right)
\delta\left( \frac{\tau^2-|\xi|^2-2\tau \angles{\sigma}_m}{2|\xi|\sigma}+ s \right ) \, ds d\sigma
 \end{align*}
 Again, changing variable
 $$r=\angles{\sigma}_m \ \Rightarrow \ r dr=\sigma d\sigma $$
  we obtain
\begin{align*}
I_+(f, g)(\tau, \xi)
&\simeq \frac1{|\xi|}\int_m^{\tau} \int_{-1}^{1} 
r(\tau-r)  f\left( \sqrt{r^2-m^2} \right)  g \left(\sqrt{ (\tau-r)^2-m^2} \right)  
\delta\left( \frac{\tau^2-|\xi|^2-2\tau r }{2|\xi| \sqrt{r^2-m^2} } + s \right ) \, ds  dr
\\
&\simeq \frac1{|\xi|}\int_{0}^\infty  
r(\tau-r) \mathbb{1}_{\mathcal D_+}(r) f\left( \sqrt{r^2-m^2} \right)  g \left( \sqrt{ (\tau-r)^2-m^2} \right) 
 dr,
 \end{align*}
 where
 \begin{align*}
\mathcal D_+:=\mathcal D_+(\tau, \xi) &=[m, \tau] \cap \left\{   r\in \R:  \  -1 \le \frac{\tau^2-|\xi|^2-2\tau r }{2|\xi| \sqrt{r^2-m^2} }\le 1 \right\} 
 \end{align*}
 
 So $r \in \mathcal D_+$ if and only if $r\in [m, \tau]$ and 
 $$ (\tau^2-|\xi|^2-2\tau r)^2 - 4|\xi|^2r^2+4m^2 |\xi|^2 \le 0. $$
The latter condition is equivalent to 
 $$
 (r-a_+)(r-a_-)\le 0,
 $$
 where 
 $$a_\pm=\frac\tau2 \pm \frac{|\xi|}2 \sqrt{\frac{\tau^2-|\xi|^2-4m^2}{\tau^2-|\xi|^2}}.
 $$
 Thus $r\in [a_-, a_+]$, and hence  $\mathcal D_+=  [m, \tau] \cap [a_-, a_+]$ .  We claim that $ [a_-, a_+] \subseteq [m, \tau]$.
 Clearly,  $a_+ \le \tau$ since $|\xi|<\tau$ and $1-\frac{
4m^2}{\tau^2-|\xi|^2} \le 1$ by \eqref{elipest}. The condition $a_-\ge m$ is equivalent to 
$$
\tau-2m\ge |\xi| \sqrt{\frac{\tau^2-|\xi|^2-4m^2}{\tau^2-|\xi|^2}}
$$
 which can be squared to obtain
 $$
 \tau^4-4m\tau^3+4m^2\tau^2-2|\xi|^2\tau^2+4m\tau |\xi|^2+ |\xi|^4\ge 0
 $$
 The expression on the left hand side can be written as 
 \begin{align*}
\left ((\tau-m)^2-m^2\right)^2-2|\xi|^2\left((\tau-m)^2-m^2\right)+ |\xi|^4=\left((\tau-m)^2-m^2)- |\xi|^2\right)^2
 \end{align*}
 which is $\ge 0.$

 Thus $\mathcal D_+= [a_-, a_+]$, and hence
 \begin{align*}
I_+(f, g)(\tau, \xi)
&\simeq \frac1{|\xi|}\int_{a_-}^{a_+}
r(\tau-r)  f\left( \sqrt{r^2-m^2} \right)  g \left( \sqrt{ (\tau-r)^2-m^2} \right)   
 dr.
 \end{align*}

\subsection{Proof of Lemma \ref{lmI+-}\eqref{lmI+-ii}}

First note that the integral $I_-(f, g)$ is supported on the set 
\begin{equation}\label{hypoid} 
\mathcal H(\tau, \xi)=\{ \eta\in \R^3: \angles{\eta}_m-\angles{\xi-\eta}_m =\tau \}.
\end{equation}
Thus for $ \eta\in \mathcal H(\tau, \xi)$ we have
\begin{equation}\label{hypest} 
\begin{split}
 |\xi|^2-\tau^2&=  |\xi|^2- (\angles{\eta}_m-\angles{\xi-\eta}_m)^2
\\
&= 2\angles{\eta}_m\angles{\xi-\eta}_m-2m^2+2\eta \cdot (\xi-\eta)) \ge 0,
\end{split}
\end{equation}

where in the last inequality we used the fact 
$$
 2\angles{\eta}_m\angles{\xi-\eta}_m\ge 2 |\eta||\xi-\eta|+ 2m^2.
$$
Since the expression on the second line is $\ge 0$, we conclude
 \begin{equation}\label{elipest} 
\tau^2 - |\xi|^2\ge 4m^2.
\end{equation}

 By \eqref{Delta2} we have
     \begin{align*}
   \delta( \tau-\angles{\eta}_m+\angles{\xi-\eta}_m)
   &=[-( \tau-\angles{\eta}_m)+ \angles{\xi-\eta}_m]
   \delta\left( -(\tau-\angles{\eta}_m)^2+
   \angles{\xi-\eta}_m^2\right)
   \\
      &=-2(\tau- \angles{\eta}_m)
   \delta\left(|\xi|^2-\tau^2+2\tau \angles{\eta}_m -2\xi \cdot \eta \right),
  \end{align*} 
  where in the first line we multiplied the argument of the delta function on the left by 
   $ -(\tau-\angles{\eta}_m)+\angles{\xi-\eta}_m$.
   
   Now introduce
 polar coordinate $\eta=\sigma \omega$, where $\omega \in \mathbb S^2$.
Proceeding similarly as in the above subsection we obtain
\begin{align*}
I_-(f, g)(\tau, \xi)
&
\simeq\frac1{|\xi|}\int_0^{\infty} \int_{-1}^{1} 
\sigma
(\angles{\sigma}_m-\tau)   f(\sigma) g\left(\sqrt{ (\angles{\sigma}_m-\tau)^2-m^2 }\right)
\delta\left( \frac{|\xi|^2-\tau^2+2\tau \angles{\sigma}_m}{2|\xi|\sigma}- s \right )\, ds d\sigma
\\
&=\frac1{|\xi|}\int_{0}^\infty
r(r-\tau)  \mathbb{1}_{ \mathcal D_-}(r)   f\left( \sqrt{r^2-m^2}\right)  g\left( \sqrt{ (r-\tau)^2-m^2} \right)  
 \, dr,
 \end{align*}
  where
 \begin{align*}
\mathcal D_-:=\mathcal D_-(\tau, \xi) &=[m, \infty) \cap  \left\{  r\in \R:  \  -1 \le \frac{|\xi|^2-\tau^2+2\tau r }{2|\xi| \sqrt{r^2-m^2} }\le 1 \right\} .
 \end{align*}

 So $r \in\mathcal D_- $ if and only if $r\in [m, \infty]$ and 
 $$ (|\xi|^2-\tau^2+2\tau r)^2 - 4|\xi|^2r^2+4m^2 |\xi|^2 \le 0 $$
The latter condition is equivalent to 
 $$
 (r-a_+)(r-a_-)\ge  0,
 $$
 where $a_\pm$ is given above.
 Thus $r\le a_-$ or $r\ge a_+$, and hence 
 $$\mathcal D_-= [m, \infty) \cap \{ (-\infty,  a_-] \cup [a_+, \infty) \}. $$
We claim that $ a_- <m $ and  $ a_+\ge m $. These would imply $\mathcal D_-= [a_+, \infty) $, and hence \begin{align*}
I_-(f, g)(\tau, \xi)
&\simeq \frac1{|\xi|}\int_{b_+}^{\infty}
r(r-\tau)  f\left( \sqrt{r^2-m^2}\right)  g\left( \sqrt{ (r-\tau)^2-m^2} \right)  
 dr.
 \end{align*}
 
 It remains to prove the claim. By \eqref{hypest} we have $-|\xi| \le \tau \le |\xi|$.
 So clearly, 
 $$
 a_-\le \frac{|\xi| } 2 \left( 1- \sqrt{1+ \frac{4m^2}{|\xi|^2-\tau^2} } \right)\le 0 \le m.
 $$
Next we
  show that $a_+\ge m$.  If  $0\le \tau \le |\xi|$ we have
  $$
a_+ \ge  \frac\tau 2+ \frac{|\xi|}2 \cdot \sqrt{\frac{4m^2}{|\xi|^2} }= \frac\tau 2 +m \ge m.
 $$
 Now assume $-|\xi| \le  \tau <0$. Then $a_+>m$ if and only if 
$$
|\xi| \sqrt{\frac{|\xi|^2-\tau^2+4m^2}{|\xi|^2-\tau^2}}\ge 2m-\tau.
$$
 Since $\tau<0$, we can square both sides to obtain the condition
 $$
 \tau^4-4m\tau^3+4m^2\tau^2-2|\xi|^2\tau^2+4m\tau |\xi|^2+ |\xi|^4\ge 0
 $$
 The expression on the left hand side can be written as 
 \begin{align*}
\left ((\tau-m)^2-m^2\right)^2-2|\xi|^2\left((\tau-m)^2-m^2\right)+ |\xi|^4=\left((\tau-m)^2-m^2)- |\xi|^2\right)^2
 \end{align*}
which is $\ge 0$.

%%%%%%%%%%%%%%%%%%%%%%%%%%%%%%%%%%%%%%%%%

\section{Proof of Lemma \ref{lmBiest} } \label{seclmBiest}

By symmetry we may assume $\lambda_1 \le \lambda_2$. 
Thus, we are reduced to the following cases:
\begin{enumerate} [(a)]
\item $\lambda_1\lesssim \lambda_2 \sim \mu$ ,
\item $\mu\ll \lambda_1\sim \lambda_2$.
\end{enumerate}
\subsection{ Case (a): $\lambda_1\lesssim \lambda_2 \sim \mu$} 
First assume $\mu =1$, and hence $\lambda_2\sim \mu=1$.
In this case we simply apply H\"{o}lder and Lemma \ref{lm-str} to obtain
 \begin{equation}\label{Biestmu1}
 \begin{split}
  \norm{P_\mu (S_m(t) f_{ \lambda_1} S_m(\pm t) g_{\lambda_2 }) }_{ L^2_{t, x} }&\lesssim 
  \norm{ S_m(t) f_{ \lambda_1} }_{  L^4_{t, x} }\norm{ S_m(\pm t) g_{\lambda_2 })}_{ L^4_{t, x} }
  \\
  &\lesssim 
\norm{ f_{\lambda_1}}_{L^2_{ x}(\R^{3})} \norm{g_{ \lambda_2}}_{L^2_{ x}(\R^{3})}.
\end{split}
 \end{equation}

 Thus, we may from now on assume $\mu > 1$. 
Taking the space-time Fourier transform we have
\begin{align*}
\mathcal{F}_{t,x}\left[P_\mu (S_m(t) f_{ \lambda_1} S_m(t) g_{\lambda_2 }) \right](\tau, \xi)&= \rho_\mu(|\xi|) \int_{\R^3}  \widehat{f_{\lambda_1} }(\eta)  \widehat{g_{\lambda_2}}(\xi-\eta)\delta( \tau-\angles{\eta}_m- \angles{\xi-\eta}_m)\, d\eta,
\\
\mathcal{F}_{t,x}\left[P_\mu (S_m(t) f_{ \lambda_1} S_m(-t) g_{\lambda_2 }) \right](\tau, \xi)&= \rho_\mu(|\xi|) \int_{\R^3}  \widehat{f_{\lambda_1} }(\eta)  \widehat{g_{\lambda_2}}(\xi-\eta)\delta( \tau-\angles{\eta}_m+ \angles{\xi-\eta}_m)\, d\eta.
\end{align*}

By Cauchy-Schwarz
\begin{align*}
& \left|\mathcal{F}_{t,x}\left[P_\mu (S_m(t) f_{ \lambda_1} S_m( \pm t) g_{\lambda_2 }) \right](\tau, \xi)\right|^2
\\
&\qquad \qquad \le I_\pm(\tau, \xi) \cdot \int_{\R^3} | \widehat{f_{\lambda_1}}(\eta) |^2   |\widehat{g_{\lambda_2}}(\xi-\eta)|^2 \delta( \tau-\angles{\eta}_m\mp \angles{\xi-\eta}_m)\, d\eta,
\end{align*}
where
\begin{align*}
 I_\pm(\tau, \xi)&=  \rho_\mu(|\xi|)  \int_{\R^3}  \rho_{\lambda_1}(|\eta|)  \rho_{\lambda_2}(|\xi-\eta|) \delta( \tau-\angles{\eta}_m\mp \angles{\xi-\eta}_m) \, d\eta.
\end{align*}
Now we claim that 
\begin{align}
\label{Ipmest}
\sup_{(\tau, \xi)\in \R^{1+3}}  I_\pm(\tau, \xi)
\lesssim
 \lambda_1^2 \quad \text{if}  \ \lambda_1\lesssim \lambda_2\sim \mu .
\end{align}
Assume for the moment that this claim holds. Then integration with respect to $\tau$ and $\xi$ gives the following:
\begin{align*}
 \|P_\mu (S_m(t) f_{ \lambda_1} S_m(\pm t) g_{\lambda_2 })\|^2 &= \int_{\R^{1+3}} \left|\mathcal{F}_{t,x}\left[P_\mu (S_m(t) f_{ \lambda_1} S_m(\pm t) g_{\lambda_2 }) \right](\tau, \xi)\right|^2 \, d\tau d\xi
\\
&\lesssim \lambda_1^2   \int_{\R^{3}} \int_{\R^3} | \widehat{f_{\lambda_1}}(\eta) |^2   |\widehat{g_{\lambda_2}}(\xi-\eta)|^2 \left( \int_{\R^{3}}\delta( \tau-\angles{\eta}_m\pm \angles{\xi-\eta}_m) \, d\tau\right) \, d\eta   d\xi
\\
& =\lambda_1^2   \|f_{\lambda_1}\|^2  \|g_{\lambda_2}\|^2,
\end{align*}
where we used the fact $ \int_{\R}\delta( \tau-\angles{\eta}_m\mp \angles{\xi-\eta}_m) \, d\tau =1.$
This estimate together with \eqref{Biestmu1} establishes  Lemma \ref{lmBiest}\eqref{Biest++} and \eqref{Biest+-} in the case $ \lambda_1\lesssim \lambda_2\sim \mu$.

Thus, it remains to prove \eqref{Ipmest}.By Lemma \ref{lmBiest} we have
\begin{align*}
 I_+(\tau, \xi)& \simeq \frac{ \rho_\mu(|\xi|) } {|\xi|} \int_{a_- }^{a_+ } 
r(\tau-r) \rho_{\lambda_1} \left( \sqrt{r^2-m^2} \right)  \rho_{\lambda_2}\left( \sqrt{ (\tau-r)^2-m^2} \right) 
  dr,
  \\
  I_-(\tau, \xi)& \simeq \frac{ \rho_\mu(|\xi|) } {|\xi|} \int_{a_+ }^{\infty } 
r(r-\tau) \rho_{\lambda_1} \left( \sqrt{r^2-m^2} \right)  \rho_{\lambda_2}\left( \sqrt{ (r-\tau)^2-m^2} \right) 
  dr.
\end{align*}

 By the support assumption in the integral
for $I_+$ we have $r\sim \angles{\lambda_1}_m$ and $\tau-r\sim \angles{\lambda_2}_m$. 
Since $\mu> 1$ and $\lambda_1\lesssim \lambda_2\sim \mu$
 we have
\begin{align*}
\sup_{(\tau, \xi)\in \R^{1+3}} I_+(\tau, \xi)& \lesssim\frac{ \angles{\lambda_1}_m \angles{ \lambda_2}_m } {\mu} \int_{ r\sim \angles{\lambda_1}_m } 
  dr\sim  \lambda_1^2.
\end{align*}

Similarly, by the support assumption in the integral
for $I_-$ we have $r\sim \angles{\lambda_1}_m$ and $r-\tau \sim \angles{\lambda_2}_m$. 
Since $\mu> 1$ and $\lambda_1\lesssim \lambda_2\sim \mu$
 we have
\begin{align*}
\sup_{(\tau, \xi)\in \R^{1+3}} I_-(\tau, \xi)& \lesssim\frac{ \angles{\lambda_1}_m \angles{ \lambda_2}_m } {\mu} \int_{ r\sim \angles{\lambda_1}_m } 
  dr\sim  \lambda_1^2.
\end{align*}

\subsection{ Case (b): $\mu\ll \lambda_1\sim \lambda_2$ }
In this case we follow the argument of Foschi–Klainerman for $m = 0$ 
\cite[Lemma 12.1]{FK00} and introduce 
 a collection of cubes $C_z=\mu z+  [0, \mu)^3 $, $z\in \Z^3$,
which induce a disjoint covering of $\R^3$. 
  By the triangle inequality 
  \begin{equation}\label{cubedecomp}
   \|P_\mu (S(t) f_{ \lambda_1} S_m( \pm t) g_{\lambda_2 })\|  \lesssim  \sum_{z, z'\in  \Z^3} \norm{ P_\mu (   S_m( t) P_{C_z}  f_{\lambda_1}  \cdot  S_m( \pm t) P_{C_{z'}}  
     g_{\lambda_2}  )},
         \end{equation}
         where $P_{C_z}$ is the frequency projection onto $C_z$. Let
          $$ f_{\lambda_1, z}:=P_{C_z}  f_{\lambda} \quad \text{and} \quad g_{\lambda_2, z'}:=P_{C_{z'}}  g_{\lambda_2}.
          $$

          Taking the Fourier Transform we have
 \begin{equation}\label{FTcubedecomp}
 \begin{split}
 &\mathcal{F}_{t,x}\left[ P_\mu (   S_m( t) f_{\lambda_1, z} \cdot  S_m(\pm  t) g_{\lambda_2, z'}  )\right](\tau, \xi)
     \\
     &\qquad \qquad = \rho_\mu(|\xi|) \int_{\R^3}  \widehat{ f_{\lambda_1, z} }(\eta)  \widehat{g_{\lambda_2, z'}}(\xi-\eta)\delta( \tau-\angles{\eta}_m\mp \angles{\xi-\eta}_m)\, d\eta.
     \end{split}
         \end{equation}
        Since $\mu\ll \lambda_1\sim \lambda_2$ and $\eta\in C_z$, $\xi-\eta\in C_{z'}$, the integral in \eqref{FTcubedecomp} yields a nontrivial contribution if $C_z$ and $C_{z'}$ are almost opposite, i.e., if $\angle{(\eta, \xi-\eta)}\sim 1$. In other words,
        for each 
     $z\in  \Z^3$, only those $z'\in  \Z^3$ with $|z+z'| \lesssim 1$ yield
    a nontrivial contribution to the sum \eqref{cubedecomp}. We use these observations and apply Lemma \ref{lm-locbiest}\eqref{locbiest++}--\eqref{locbiest+-} below to \eqref{cubedecomp}, and use Cauchy-Schwarz to obtain 
      \begin{align*}
     \|P_\mu (S_m(t) f_{ \lambda_1} S_m( t) g_{\lambda_2 })\|  &\lesssim \mu \sum_{ |z+z'|\lesssim 1 } \|f_{z, \lambda_1}\|  \|g_{z', \lambda_2}\|
     \\
      &\lesssim  \mu \left(\sum_{ z\in \Z^3} \| f_{z, \lambda_1}\|^2 \right)^\frac12  \left(\sum_{ z' \in \Z^3} \| g_{z', \lambda_2}\|^2 \right)^\frac12
      \\
      &\sim \mu  \| f_{ \lambda_1}\| \| g_{ \lambda_2}\|
\end{align*}  
    and 
       \begin{align*}
     \|P_\mu (S_m(t) f_{ \lambda_1} S_m( -t) g_{\lambda_2 })\|  &\lesssim (\mu\lambda_1)^\frac12 \sum_{ |z+z'|\lesssim 1 } \|f_{z, \lambda_1}\|  \|g_{z', \lambda_2}\|
     \\
      &\lesssim   (\mu\lambda_1)^\frac12  \left(\sum_{ z\in \Z^3} \| f_{z, \lambda_1}\|^2 \right)^\frac12  \left(\sum_{ z' \in \Z^3} \| g_{z', \lambda_2}\|^2 \right)^\frac12
      \\
      &\sim  (\mu\lambda_1)^\frac12 \| f_{ \lambda_1}\| \| g_{ \lambda_2}\|.
\end{align*}

    \begin{lemma}[Refined bilinear estimates]\label{lm-locbiest}
    Assume $ \ \mu \ll \lambda_1\sim \lambda_2$. For all $(z,z')\in \Z^3 \times \Z^3$ we have the following localized bilinear estimates:
    \begin{enumerate}[(i)]
    \item \label{locbiest++} (++) interaction:
             \begin{equation}
              \label{loc-biest++}
   \norm{ P_\mu (   S_m( t)f_{z, \lambda_1} \cdot  S_m(  t) g_{z', \lambda_2}  )} \lesssim \mu   \|f_{z, \lambda_1}\| \|g_{z', \lambda_2}\| .
        \end{equation}
      \item (+-)\label{locbiest+-} interaction:
       \begin{equation}\label{loc-biest+-}
    \norm{ P_\mu (   S_m( t) f_{z, \lambda_1}  \cdot  S_m( - t) g_{z', \lambda_2} )} \lesssim (\mu\lambda_1)^\frac12   \|f_{z, \lambda_1}\| \|g_{z', \lambda_2}\| .
            \end{equation}

    \end{enumerate}
    \end{lemma}  
      
      \subsubsection{Proof of Lemma \ref{lm-locbiest}\eqref{locbiest++} }
Set $
\rho_{z, \lambda}(|\xi|)  = \mathbb{1}_{B_{2 \mu} (\mu z)} (\xi)  \cdot \rho_{ \lambda}(|\xi|) ,
$
where $B_{2 \mu} (\mu z)$ denotes the ball of center $\mu z$ and radius $2\mu$. 
Squaring \eqref{FTcubedecomp} and using Cauchy-Schwarz we have
\begin{equation}\label{locbiest++int}
\begin{split}
& \left|\mathcal{F}_{t,x}\left[ P_\mu (   S_m( t) f_{\lambda_1, z} \cdot  S_m( t)  g_{\lambda_2, z'}  ) \right](\tau, \xi)\right|^2
\\
&\qquad \qquad \lesssim J^+_{z,z'}(\tau, \xi) \cdot \int_{\R^3} | \widehat{f_{z, \lambda_1}}(\eta) |^2   |\widehat{g_{z', \lambda_2}}(\xi-\eta)|^2 \delta( \tau-\angles{\eta}_m- \angles{\xi-\eta}_m)\, d\eta,
\end{split}
\end{equation}
where
\begin{align*}
J^+_{z,z'}(\tau, \xi)&=  \rho_\mu(|\xi|)  \int_{\R^3}  \rho_{z, \lambda_1}(|\eta|)  \rho_{z', \lambda_2}(|\xi-\eta|) \delta( \tau-\angles{\eta}_m- \angles{\xi-\eta}_m) \, d\eta.
\end{align*}
     
 It suffices to show for all $(z,z')\in \Z^3 \times \Z^3$ that
\begin{equation}
\label{J+est}
 \sup_{ (\tau, \xi)\in \R^{1+3}} J^+_{z,z'}(\tau, \xi)
\lesssim
 \mu^2 \quad \text{if}  \ \mu \ll \lambda_1\sim \lambda_2.
\end{equation}
Integration of \eqref{locbiest++int} in $\tau$ and $\xi$ then yields Lemma \ref{lm-locbiest}\eqref{locbiest++}.

We now prove \eqref{J+est}. By \eqref{Delta1} we have
\begin{align*}
J^+_{z,z'}(\tau, \xi) &=\rho_\mu(|\xi|) \int_{\eta\in \mathcal E(\tau, \xi)  \cap B_{2 \mu} (\mu z) } \frac{ \rho_{ \lambda_1}(|\eta|)    \rho_{ z', \lambda_2}(|\xi-\eta|) }{ |  \nabla_\eta (\angles{\eta}_m+\angles{\xi-\eta}_m )  | } dS\eta,
\end{align*}
where the set
$
\mathcal E(\tau, \xi)$ is as in \eqref{elipsoid}.
Since $\angle{(\eta, \xi-\eta)}\sim 1$ we have
$$
|\nabla_\eta (\angles{\eta}_m+\angles{\xi-\eta}_m ) | =\left|  \frac{\eta}{\angles{\eta}_m}- \frac{\xi-\eta}{\angles{\xi-\eta}_m}\right|\sim 1.
$$
The domain of integration, $  \mathcal E(\tau, \xi) \cap B_{2 \mu} (\mu z) $, is a two dimensional surface with area $\lesssim \mu^2$. Thus for all  $(z, z')\in \Z^3 \times \Z^3$ and $(\tau, \xi)\in \R^{1+3}$, we have 
$$
J^+_{z,z'}(\tau, \xi)  \lesssim |\mathcal E(\tau, \xi) \cap B_{2 \mu} (\mu z) | \lesssim \mu^2,
$$
and this establishes the $\eqref{J+est}$  .  
     
      \subsubsection{Proof of Lemma \ref{lm-locbiest}\eqref{locbiest+-} }
Here we follow  Foschi–Klainerman for $m=0$
\cite[proof of Lemma 12.1-- equation (66)]{FK00}). To estimate the left hand side of \eqref{loc-biest+-} first we square  \eqref{FTcubedecomp} (the +- case), write it as a double integral and then integrate over $\tau$ and $\xi$ . After applying the Fubini-Tonelli theorem and rearranging the integrand we obtain
\begin{align*}
&\norm{ P_\mu (   S_m( t) f_{z, \lambda_1}  \cdot  S_m( - t) g_{z', \lambda_2} )}^2 
\\
 & \qquad = \int_{\R^{9}}   \widehat{ f_{\lambda_1, z} }(\eta) \widehat{g_{\lambda_2, z'}}(\zeta) \cdot \widehat{ f_{\lambda_1, z} }(\xi-\zeta) \widehat{g_{\lambda_2, z'}}(\xi-\eta)   \, d\sigma ( \xi, \eta, \zeta),
\end{align*}
where $d\sigma ( \xi, \eta, \zeta)$ is the surface measure
\begin{align*}
 d\sigma ( \xi, \eta, \zeta) &=  \rho^2_\mu(|\xi|)\cdot 
\delta\left(\angles{\eta}_m+ \angles{\zeta}_m -\angles{\xi-\eta}_m
 -\angles{\xi-\zeta}_m \right)
 d\xi d\eta d\zeta.
\end{align*}
Applying Cauchy-Schwarz on the terms $  \widehat{ f_{\lambda_1, z} }(\eta) \widehat{g_{\lambda_2, z'}}(\zeta) $ and $\widehat{ f_{\lambda_1, z} }(\xi-\zeta) \widehat{g_{\lambda_2, z'}}(\xi-\eta) $ with respect to the measure 
 $d\sigma ( \xi, \eta, \zeta)$, and then changing variables we obtain
\begin{align*}
\norm{ P_\mu (   S_m( t) f_{z, \lambda_1}  \cdot  S_m( - t) g_{z', \lambda_2} )}^2 &\le  \int_{\R^{9}}   |\widehat{ f_{\lambda_1, z} }(\xi-\eta)|^2 | \widehat{g_{\lambda_2, z'}}(\xi-\zeta) |^2 \, d\sigma ( \xi, \eta, \zeta)
\\
&\lesssim \int_{\R^{6}} J^-_{z,z'} (\eta, \zeta) \cdot      |\widehat{ f_{\lambda_1, z} }(\xi-\eta)|^2 | \widehat{g_{\lambda_2, z'}}(\xi-\zeta) |^2 \, d\eta d\zeta,
\end{align*}
where
\begin{align*}
J^-_{z,z'}(\eta, \zeta) 
&= \int_{\R^3}   \rho_{z, \lambda_1}(|\xi-\eta|)  \rho_{z', \lambda_2}(|\xi-\zeta|) \rho_\mu(|\xi|) 
\\
& \qquad \qquad \qquad \qquad \times \delta (\angles{\eta}_m+ \angles{\zeta}_m- \angles{\xi-\eta}_m  -\angles{\xi-\zeta}_m) \, d\xi .
\end{align*}

So it suffices to show for all $(z,z')\in \Z^3 \times \Z^3$ that
\begin{equation}
\label{J-est}
 \sup_{ \eta \in C_z,  \ \zeta \in C_{z'}} J^-_{z,z'}(\eta, \zeta)
\lesssim
 \mu \lambda_1 \quad \text{for}  \ \mu \ll \lambda_1\sim \lambda_2.
\end{equation}

By \eqref{Delta1} we have
\begin{align*}
J^-_{z,z'}(\eta, \zeta)=  \int_{\xi \in \mathcal E(\eta, \zeta)} \frac{\rho_{z, \lambda_1}(|\xi-\eta|)  \rho_{z', \lambda_2}(|\xi-\zeta|) \rho_\mu(|\xi|) }{ |  \nabla_\xi (\angles{\xi-\eta}_m  + \angles{\xi-\zeta}_m)  | }   dS_\xi,
\end{align*}
where 
$$
\mathcal  E(\eta, \zeta)= \left\{ \xi\in \R^3 : \ \angles{\xi-\eta}_m  +\angles{\xi-\zeta}_m= \angles{\eta}_m+ \angles{\zeta}_m \right\}.
$$

Now we compute
\begin{align*}
 |\nabla_\xi (\angles{\xi-\eta}_m  + \angles{\xi-\zeta}_m)  |^2 &=\left|  \frac{\xi-\eta}{\angles{\xi-\eta}_m}+ \frac{\xi-\zeta}{\angles{\xi-\zeta}_m}\right|^2 
\\
&=\left|  \frac{|\xi-\eta|}{\angles{\xi-\eta}_m}- \frac{|\xi-\zeta|}{\angles{\xi-\zeta}_m}\right|^2 + \frac{ 2\left[|\xi-\eta|  |\xi-\zeta|+ (\xi-\eta)\cdot (\xi-\zeta)\right]}{\angles{\xi-\eta}_m \angles{\xi-\zeta}_m}
\\
&\gtrsim  \theta^2,
\end{align*}
where  $\theta=\angle\left(\xi-\eta, -(\xi-\zeta)\right).$ Observe that since $\xi-\eta \in C_z$ and $ \xi-\zeta \in C_{z'}$, where $|z+z'|\lesssim 1$ (see the comments under equation \eqref{FTcubedecomp}),  we conclude that
$$
\theta\sim \mu/\lambda_1.
$$
Thus $|\nabla_\xi (\angles{\xi-\eta}_m  + \angles{\xi-\zeta}_m)  |\gtrsim  \mu/\lambda_1$, and hence
\begin{align*}
J^-_{z,z'}(\eta, \zeta)\lesssim \frac{\lambda_1}{\mu} \int_{\xi \in \mathcal E(\eta, \zeta)  \cap B_{2\mu} (0)}   dS_\xi = \frac{\lambda_1}{\mu}  |\mathcal E(\eta, \zeta)  \cap B_{2\mu} (0)| \lesssim \mu\lambda_1
\end{align*}
since $  \mathcal E(\tau, \xi) \cap B_{2 \mu} (\mu z) $ is a two dimensional surface with area $\lesssim \mu^2$. This establishes \eqref{J-est}.

%*******************************************************************

\subsection*{Acknowledgement}
 The author would like to thank Sigmund Selberg
 for his encouragement and useful discussions while working on the paper. The author is also grateful for the anonymous referees’ valuable comments
and suggestions.
%*****************************************************************

%**************************************************************
%*****************************************************************

\bibliographystyle{plain} \bibliography{3DDirac.bbl}

\end{document}